\documentclass[11pt]{article}
\usepackage{latexsym, amssymb, amsmath, amscd, amsthm, verbatim}
\usepackage{enumerate}
\usepackage[latin1]{inputenc}
\usepackage{amstext}
\usepackage{color}
\usepackage{graphicx}
\usepackage{longtable}
\usepackage{calc}

\numberwithin{equation}{section}
\numberwithin{table}{section}
\numberwithin{figure}{section}
\newtheorem{theorem}{Theorem}[section]

\newtheorem{lemma}[theorem]{Lemma}
\newtheorem{corollary}[theorem]{Corollary}

\newtheorem{definition}[theorem]{Definition}
\theoremstyle{definition}
\newtheorem{example}[theorem]{Example}
\newtheorem{remark}[theorem]{Remark}

\makeatletter
\def\revddots{\mathinner{\mkern1mu\raise\p@
     \vbox{\kern7\p@\hbox{.}}\mkern2mu
     \raise4\p@\hbox{.}\mkern2mu\raise7\p@\hbox{.}\mkern1mu}}
\makeatother

\newcommand {\mat}  [1] {\left[\begin{array}{#1}}
\newcommand {\rix}      {\end{array}\right]}

\newcommand{\la}{\ensuremath{\lambda}}

\def\max{\mathop{\rm max}}
\def\rank{\mathop{\rm rank}}
\def\min{\mathop{\rm min}}

\def\det{\mathop{\rm det}}

\mathchardef\Gamma="7100
\mathchardef\Delta="7101
\mathchardef\Theta="7102
\mathchardef\Lambda="7103
\mathchardef\Xi="7104
\mathchardef\Pi="7105
\mathchardef\Sigma="7106
\mathchardef\Upsilon="7107
\mathchardef\Phi="7108
\mathchardef\Psi="7109
\mathchardef\Omega="710A

\newcommand{\bF}{\mathbb{F}}

\newcommand{\FF}{\mathbb{F}}
\newcommand{\RR}{\mathbb{R}}
\newcommand{\CC}{\mathbb{C}}
\newcommand{\BF}{\mathbb{F}}
\newcommand{\beq}{\begin{equation}}
\newcommand{\eeq}{\end{equation}}

\mathcode`@="8000 
{\catcode`\@=\active\gdef@{\mkern1mu}}

\def\veps{\varepsilon}

\begin{document}

\title{Robustness and Perturbations of Minimal Bases}
\author{Paul Van Dooren\footnote{Department of Mathematical Engineering, Universit\'e catholique de Louvain, Avenue Georges Lema\^{i}tre 4, B-1348 Louvain-la-Neuve, Belgium.  Email: {\tt paul.vandooren@uclouvain.be}. Supported by the Belgian network DYSCO (Dynamical Systems, Control, and Optimization), funded by the Interuniversity Attraction Poles Programme initiated by the Belgian Science Policy Office.} \hspace{0.25cm} and \hspace{0.25cm} Froil\'an M. Dopico\footnote{Departamento de Matem\'aticas, Universidad Carlos III de Madrid, Avenida de la Universidad 30, 28911, Legan\'es, Spain.  Email: {\tt dopico@math.uc3m.es}. Supported by ``Ministerio de Econom\'{i}a, Industria y Competitividad of Spain'' and ``Fondo Europeo de Desarrollo Regional (FEDER) of EU'' through grants MTM-2015-68805-REDT, MTM-2015-65798-P (MINECO/FEDER, UE).}
               }
\date{\today}
\maketitle

\begin{abstract}
Polynomial minimal bases of rational vector subspaces are a classical concept that plays an important role in control theory, linear systems theory, and coding theory. It is a common practice to arrange the vectors of any minimal basis
as the rows of a polynomial matrix and to call such matrix simply a minimal basis.
Very recently, minimal bases, as well as the closely related pairs of dual minimal bases, have been applied to a number of problems that include the solution of general inverse eigenstructure problems for polynomial matrices, the development of new classes of linearizations and $\ell$-ifications of polynomial matrices, and backward error analyses of complete polynomial eigenstructure problems solved via a wide class of strong linearizations. These new applications have revealed that although the algebraic properties of minimal bases are rather well understood, their robustness and the behavior of the corresponding dual minimal bases under perturbations have not yet been explored in the literature, as far as we know. Therefore, the main purpose of this paper is to study in detail when a minimal basis $M(\la)$ is robust under perturbations, i.e., when all the polynomial matrices in a neighborhood of $M(\la)$ are minimal bases, and, in this case, how perturbations of $M(\la)$ change its dual minimal bases. In order to study such problems, a new characterization of whether or not a polynomial matrix is a minimal basis in terms of a finite number of rank conditions is introduced and, based on it, we prove that polynomial matrices are generically minimal bases with very specific properties. In addition, some applications of the results of this paper are discussed.
\end{abstract}

{\small

 }


\section{Introduction} \label{sect.intro}

Minimal bases, formed by vectors with polynomial entries, of a rational vector subspace were made popular in standard references of control theory and linear systems theory as those of Wolovich \cite{wolovich}, Forney \cite{For75}, and Kailath \cite{Kai80}, although all three of them cite earlier work for some theoretical developments on the so-called {\em minimal polynomial bases}. For instance, one can  read in \cite[p. 460]{Kai80} the following sentence: {\em I.C. Gohberg pointed out to the author that the significance of minimal bases was perhaps first realized by J. Plemelj in 1908 and then substantially developed in 1943 by N.I. Mushkelishvili and N.P. Vekua.} This means that this paper deals with an almost 110 years old classical mathematical notion. However, the discovery of the importance of this concept in applications had to wait until the 1970s, when the contributions of authors such as Wolovich, Forney, Kailath, and others, provided computational schemes for constructing a minimal basis from an arbitrary polynomial basis,
and showed the key role that this notion plays in multivariable linear systems.
These systems could be modeled by rational matrices, polynomial matrices,
or linearized state-space models, and had tremendous potential for solving analysis and design problems in control theory as well as in coding theory. The reader is referred to \cite{For75,Kai80} and the references therein for more information on minimal bases and their applications, and also to the brief revision included in the Section \ref{Sec:Invariants} of this paper. Moreover, many papers have been published after \cite{Kai80} on the computation and applications of minimal bases and some of them can be found in the references included in \cite{antoniou-min-bases}.

Very recently, minimal bases, and the closely related notion of pairs of dual minimal bases, have been applied to the solution of some problems that have attracted the attention of many researchers in the last fifteen years. For instance, minimal bases have been used (1) in the solution of inverse complete eigenstructure problems for polynomial matrices (see \cite{DDMV,DDV} and the references therein), (2) in the development of new classes of linearizations and $\ell$-ifications of polynomial matrices \cite{buenoREU,DDV-l-ifications,blockKron,BMBelifications,fassbendersaltenberger,lawrence-perez-cheby,robol},
which has allowed to recognize that many important linearizations commonly used in the literature are constructed via dual minimal bases (including the classical Frobenius companion forms), (3) in the explicit construction of linearizations of rational matrices \cite{amparanjointrat},
and (4) in the backward error analysis of complete polynomial eigenvalue problems solved via the so-called ``block Kronecker linearizations'' of polynomial matrices \cite[Section 6]{blockKron}, which include the interesting class of Fiedler linearizations (see \cite{buenoREU,blockKron} for references on this class of linearizations), but {\em do not include most of the linearizations and $\ell$-ifications that can be constructed from minimal bases} \cite{buenoREU,DDV-l-ifications,blockKron,BMBelifications,fassbendersaltenberger,lawrence-perez-cheby,robol}.
See also \cite{lawrence-vanbarel-vandooren} for additional references on the role played by dual minimal bases in backward error analyses.

The backward error analysis in \cite[Section 6]{blockKron} uses heavily the following two key ideas: (1) that the very particular minimal bases of degree one which are blocks of the block Kronecker linearizations are robust under perturbations, in the sense that all the polynomial matrices in a neighborhood of these minimal bases are also minimal bases with similar properties; and (2) that if these particular minimal bases are perturbed by a certain magnitude, then a particular choice of their dual minimal bases changes essentially by the same magnitude. These results have made clear that for extending the backward error analysis in \cite[Section 6]{blockKron} to much more general contexts that include many other classes of linearizations and also $\ell$-ifications of polynomial matrices, it is necessary to study in depth the robustness of general minimal bases and the behaviour of their dual minimal bases under perturbations, questions that have not been considered so far in the literature.

To solve the robustness and perturbation problems for minimal bases discussed in the previous paragraph is the main goal of this paper, which is achieved in Sections \ref{sec.smoothness} and \ref{sec.dualminbases}. We emphasize that the solution of these problems is based on a number of additional results that, in our opinion, are by themselves interesting contributions to the theory and applications of minimal bases. For instance, a new characterization of minimal bases in terms of Sylvester matrices is presented in Section \ref{sec.newchar}, and we prove in Section \ref{sec.genericity} that polynomial matrices are generically minimal bases with very specific properties that are encoded in the concept of polynomial matrices with full-Sylvester-rank, which are studied in Section \ref{sect:RankProp}. In Section \ref{sec.classrevisited}, the standard rank characterization of minimal bases is connected to the results in this paper. This work is completed with a discussion of some preliminary applications in Section \ref{sec.applications} and the conclusions are presented in Section \ref{sect.conclusions}.

\section{Preliminaries}
\label{Sec:Invariants}
The results in Sections \ref{Sec:Invariants}, \ref{sec.newchar}, and \ref{sect:RankProp} of this paper hold for an arbitrary field $\FF$, while in the remaining sections $\FF$ is the field of real numbers $\RR$ or of complex numbers $\CC$, which will be simultaneously denoted by $\FF$.
We adopt standard notations in the area:
$\FF[\la]$ denotes the ring of polynomials in the variable $\la$
with coefficients in $\FF$ and $\FF(\la)$ denotes the field of fractions of $\FF[\la]$, also known as the field of rational functions over $\bF$.
Vectors with entries in $\FF[\la]$ are called polynomial vectors. In addition, $\FF[\la]^{m\times n}$ stands for the set of $m\times n$ polynomial matrices, and $\FF(\la)^{m\times n}$ for the set of $m\times n$ rational matrices. The {\em degree} of a polynomial vector, $v(\la)$, or matrix, $P(\la)$, is the highest degree of all its entries and is denoted by $\deg(v)$ or $\deg(P)$. The rigorous definition of genericity we adopt requires to work inside the vector space over $\FF$ of $m \times n$ polynomial matrices with degree at most $d$, which is denoted by $\FF[\la]^{m\times n}_d$. Finally, $\overline \FF$ denotes the algebraic closure of $\FF$, $I_n$ the $n\times n$ identity matrix,
and $0_{m\times n}$ the $m\times n$ zero matrix, where the sizes will be omitted when they are clear from the context.

Polynomial matrices with size $m\times n$ and degree at most $d$ are represented in this paper in the monomial basis as $P(\la) = \sum_{i=0}^d P_i \la^i$, where $P_i \in \FF^{m\times n}$. If $P_d \ne 0$, then the  degree of $P(\la)$ is precisely $d$, that is $\mbox{deg} (P) =d$.
The rank of $P(\la)$ (sometimes called ``normal rank'') is just the rank of $P(\la)$ considered as a matrix over the field $\FF (\la)$,
and is denoted by $\mbox{rank} (P)$.
The finite eigenvalues of $P(\la)$ are the roots of its invariant polynomials,
and associated to each such eigenvalue are elementary divisors of $P(\la)$. These and the rest of concepts used in this paper can be found in the classical books \cite{Gan59,GohbergLancasterRodman09,Kai80}, as well as in the summaries included in \cite[Sect. 2]{DDM} and in \cite[Sect. 2]{DDMV}, which are recommendable since are brief and follow exactly the conventions in this paper.

Since ``minimal basis'' is the key concept of this paper, we revise its definition, characterization, main properties, and related notions. It is well known that every rational vector subspace $\mathcal{V}$,
i.e., every subspace $\mathcal{V} \subseteq \FF(\la)^n$ over the field $\FF(\la)$,
has bases consisting entirely of polynomial vectors. Among them some are minimal in the following sense introduced by Forney \cite{For75}: a {\em minimal basis} of $\mathcal{V}$ is a basis of $\mathcal{V}$ consisting of polynomial vectors whose sum of degrees is minimal among all bases of $\mathcal{V}$  consisting of polynomial vectors. The fundamental property \cite{For75,Kai80} of such bases is that the ordered list of degrees of the polynomial vectors in any minimal basis of $\mathcal{V}$ is always the same. Therefore, these degrees are an intrinsic property of the subspace $\mathcal{V}$ and are called the {\em minimal indices} of $\mathcal{V}$.
This discussion immediately leads us to the definition of the minimal bases and indices of a polynomial matrix. An $m\times n$ polynomial matrix $P(\lambda)$ with rank $r$ smaller than $m$ and/or $n$ has non-trivial left and/or right rational null-spaces, respectively,
over the field $\FF (\la)$:
\begin{eqnarray*}
  {\cal N}_{\ell}(P)\!\! &:=& \!\left\{y(\la)^T \in\BF(\la)^{1 \times m}
                           \,:\,y(\la)^T P(\la)\equiv0^T\right\},\\
                             {\cal N}_r(P)\!\! &:=& \!\left\{x(\la)\in\BF(\la)^{n \times 1}
                           \,:\, P(\la)x(\la)\equiv0\right\}.
\end{eqnarray*}
Polynomial matrices with non-trivial left and/or right null-spaces
are called {\em singular} polynomial matrices. If the rational subspace ${\cal N}_{\ell}(P)$ is non-trivial, it has minimal bases and minimal indices, which are called the {\em left minimal bases and indices} of $P(\la)$. Analogously, the {\em right minimal bases and indices} of $P(\la)$ are those of ${\cal N}_{r}(P)$, whenever this subspace is non-trivial.

The definition of minimal basis given above cannot be easily handled in practice. Therefore, \cite[p. 495]{For75} includes five equivalent characterizations of minimal bases. Among those, we present in Theorem \ref{minbasis_th} only the one we believe is the most useful in practice, since relies on the ranks of constant matrices. The statement of Theorem \ref{minbasis_th} requires to introduce first Definition \ref{colred}.
For brevity, we use the expression ``column (resp., row) degrees'' of a polynomial matrix to mean the degrees of its column (resp., row) vectors.

\begin{definition}\label{colred}
Let $d'_1,\ldots,d'_n$ be the column degrees of $N(\lambda) \in \FF[\la]^{m\times n}$.
The highest-column-degree coefficient matrix of $N(\la)$, denoted by $N_{hc}$,
is the $m\times n$ constant matrix whose $j$th column
is the vector coefficient of $\lambda^{d'_j}$ in the $j$th column of $N(\lambda)$.
The polynomial matrix $N(\la)$ is said to be \emph{column reduced} if $N_{hc}$ has full column rank.

Similarly, let $d_1,\ldots,d_m$ be the row degrees of $M(\la) \in \FF[\la]^{m\times n}$. The highest-row-degree coefficient matrix of $M(\la)$, denoted by $M_{hr}$,
is the $m\times n$ constant matrix whose $j$th row is the vector coefficient
of $\lambda^{d_j}$ in the $j$th row of $M(\lambda)$.
The polynomial matrix $M(\la)$ is said to be \emph{row reduced} if $M_{hr}$ has full row rank.
\end{definition}

Theorem \ref{minbasis_th} provides the announced characterization of minimal bases proved in \cite{For75}.

\begin{theorem}\label{minbasis_th}
  The columns {\rm (}resp., rows{\rm )} of a polynomial matrix $N(\lambda)$ over a field $\FF$
  are a minimal basis of the subspace they span
  if and only if $N(\lambda_0)$ has full column {\rm (}resp., row{\rm )} rank
  for all $\lambda_0 \in \overline \FF$,
  and $N(\la)$ is column {\rm (}resp., row{\rm )} reduced.
\end{theorem}

\begin{remark} \label{rem:NEWconvention}
In this paper we follow the convention in \cite{For75} and often say, for brevity, that a $p \times q$ polynomial matrix $N(\la)$ is a {\em minimal basis} if the columns (when $q<p$) or rows (when $p<q$) of $N(\la)$ form a minimal basis of the rational subspace they span.
\end{remark}

\begin{remark} \label{rem:unicity}
If $M(\la)\in \bF[\lambda]^{m\times k}$ is a row reduced polynomial matrix, then $M(\la)$ has full row (normal) rank. This can be seen as follows: let $d_i, i=1,\ldots,m$, be the row degrees of $M(\la)$ and let $S(\la)$ be an $m\times m$ submatrix of $M(\la)$ such that the corresponding submatrix of $M_{hr}$, denoted by $S_{hr}$, is nonsingular. Then, $\det S(\la) = \la^{d_1+\cdots+d_m} \, \det S_{hr} + \mbox{(lower degree terms in $\la$)} \ne 0$. This means that the rows of $M(\la)$ form a basis of the rational subspace $\mathcal{V}$ they span (which is the row space of $M(\la)$) and, by definition of minimal basis, the sum of its row degrees $\sum_{i=1}^m d_i$ is an upper bound for the sum of the minimal indices $\sum_{i=1}^m \eta_i$ of $\mathcal{V}$. Moreover, equality of the sums implies that $M(\la)$ is a minimal basis of $\mathcal{V}$ and that the ordered lists $\{d_i\}$ and $\{\eta_i\}$ are also equal, by the uniqueness of the minimal indices.
\end{remark}

Next, we introduce the concept of {\em dual minimal bases}, which has played a key role in a number of recent applications \cite{DDMV,DDV-l-ifications,blockKron,lawrence-perez-cheby,robol}.

\begin{definition}\label{dual-m-b}
  Polynomial matrices $M(\lambda)\in \bF[\lambda]^{m\times k}$
  and $N(\lambda)\in \bF[\lambda]^{n\times k}$ with full row ranks
  are said to be {\em dual minimal bases}
  if they are minimal bases satisfying $m+n=k$ and $M(\lambda) \, N(\lambda)^T=0$.
\end{definition}

The name ``dual minimal bases'' seems to be very recent in the literature, because \cite {DDMV} is the first reference where is used, to our knowledge. However, the origins of such concept can be found in \cite[Section 6]{For75}. According to \cite[Section 6]{For75}, dual minimal bases span rational vector subspaces of $\FF (\la)^k$ that are dual to each other. In the language of null-spaces of matrix polynomials, observe that $M(\lambda)$ is a minimal basis of ${\cal N}_\ell (N(\la)^T)$
and that $N(\lambda)^T$ is a minimal basis of ${\cal N}_r (M(\la))$.
As a consequence, the right minimal indices of $M(\la)$ are the row degrees of $N(\la)$ and the left minimal indices of $N(\la)^T$ are the row degrees of $M(\la)$.
Note that dual minimal bases have been defined with more columns than rows,
as in the classical reference \cite{For75}, although, one can also use matrices with more rows than columns in the definition. It follows from this discussion that for every minimal basis, there exists a minimal basis
that is dual to it. In addition, every minimal basis is a minimal basis of some matrix polynomial.

The next theorem is crucial for the rest of this paper. Its first part was proven in \cite{For75}, while the second (converse) part has been proven very recently in \cite{DDMV}.
\begin{theorem} \label{thm:dualbasis}
  Let $M(\la)\in \bF[\lambda]^{m\times (m + n)}$ and $N(\lambda) \in \bF[\lambda]^{n\times (m+n)}$
  be dual minimal bases with row degrees $(\eta_1,\hdots,\eta_m )$
  and $(\varepsilon_1,\hdots,\veps_{n} )$, respectively.
  Then
\begin{equation} \label{eqn.keyequality}
  \sum_{i=1}^m \eta_i \,=\, \sum_{j=1}^{n}\varepsilon_j \,.
\end{equation}
Conversely, given any two lists of nonnegative integers
$(\eta_1,\hdots,\eta_m)$ and $(\varepsilon_1,\hdots,\veps_{n})$
satisfying \eqref{eqn.keyequality},
there exists a pair of dual minimal bases $M(\la)\in \bF[\lambda]^{m\times (m+n)}$
and $N(\lambda) \in \bF[\lambda]^{n\times (m+n)}$
with precisely these row degrees, respectively.
\end{theorem}

\section{A finite number of rank conditions for minimal bases} \label{sec.newchar}

The characterization of a minimal basis given in Theorem \ref{minbasis_th} does not seem tractable at first sight from a numerical point of view, since it requires a rank test over all $\lambda_0 \in \overline \FF$. In this section we show that this can be reduced to a finite number of rank tests on matrices that can be described
from the dimensions and coefficients of $M(\la)$. A crucial role will be played here by the so-called Sylvester matrices \cite{AndJuryIEEE76,BKAK} of a polynomial matrix $M(\la)$ of degree at most $d$,
 \begin{equation}\label{eq:Md}
   M(\la)  := M_0 + M_1 \la + \cdots +  M_d \la^d, \quad M_i\in \FF^{m\times (m+n)},
   \end{equation}
  defined for $k=1,2,\ldots$ as
  \begin{equation}\label{eq:Sylvester}
   S_k := \underbrace{ \left[ \begin{array}{ccccc} M_0 \\ M_1 & M_0 \\
   \vdots & M_1 &  \ddots \\
   M_d & & \ddots & M_0 \\
   0 & M_d & & M_1\\
   \vdots & \ddots & \ddots & \vdots \\
   0 & \ldots & 0 & M_d
   \end{array}\right]}_{k \; \mathrm{blocks}}, \quad S_k \in  \FF^{(k+d)m\times k(m+n)},
  \end{equation}
where $S_k$ has $k$ block columns. In order to avoid trivialities, we assume throughout the paper that $m>0, n>0,$ and $d > 0$. In certain results involving Sylvester matrices of several polynomial matrices, we will use the notation $S_k (M)$ to indicate that the Sylvester matrices are associated to the polynomial matrix $M(\la)$.

The ranks of the matrices $S_k$ in \eqref{eq:Sylvester} are fundamental in this paper and the following two simple lemmas are used very often.
\begin{lemma} \label{lk}
Let $S_k$ for a given index $k>1$ have full column rank $r_k:=k(m+n)$. Then all submatrices $S_\ell$ with $1 \le \ell < k$
also have full column rank $r_\ell:=\ell(m+n)$.
\end{lemma}
\begin{proof}
Each matrix $S_\ell$ with $1 \le \ell < k$, appropriately padded with zeros, forms the first $\ell$ block columns of $S_k$, from which the rank condition trivially follows.
\end{proof}

\begin{lemma}  \label{kl}
Let  $S_k$ with a given index $k>0$  have full row rank $r_k:=(k+d)m$. Then all embedding matrices $S_\ell$ with $k<\ell$ also have full row rank
$r_\ell := (\ell+d)m$.
\end{lemma}
\begin{proof}
We first prove that the result holds for $\ell=k+1$. Since $S_k$ has full row rank $r_k$, its bottom block row $\left[\begin{array}{c|c} 0 & M_d \end{array}\right]$ has full row rank $m$ and so does the matrix $M_d$. Since the matrix $S_{k+1}$ can be partitioned in a block triangular form
$$S_{k+1} = \left[\begin{array}{c|c} S_k & X \\  \hline 0 & M_d \end{array}\right], $$
where both $S_k$ and $M_d$ have full row rank, $S_{k+1}$ has also full row rank. The result for general $\ell$ larger than $k$, then easily follows by induction.
\end{proof}

In some of the results in this section, we will assume that $M(\la) \in \FF[\la]^{m \times (m+n)}$ has (full) normal rank $m$ so that it has a (right)
null-space of dimension $n$ (over the field of rational functions). The connection of this null-space $\mathcal{N}_r (M)$ with the Sylvester matrices in \eqref{eq:Sylvester} comes from the fact that any $N(\la)\in \FF[\la]^{n\times (m+n)}$ such that the columns of $N(\la)^T$ belong to $\mathcal{N}_r (M)$ (i.e.\/ $M(\la)N(\la)^T=0$) with expansion
 \begin{equation}\label{eq:Nd}
   N(\la)  := N_0 + N_1 \la + \cdots +  N_{s} \la^{s}, \quad N_i\in \FF^{n\times (m+n)},
   \end{equation}
will satisfy the equations
$$
S_k \left[ \begin{array}{c} N^T_0 \\  \vdots \\ N^T_{s} \\ 0  \end{array}\right] = 0 \quad \mbox{for all} \quad  k \geq s+1 .$$
In particular, one can choose $N(\la)^T$ to be a minimal basis of $\mathcal{N}_r (M)$ and then to derive certain rank conditions from this. The following theorem is proven in \cite{BKAK} using this type of arguments.

\begin{theorem} \label{thm:ranktheorem} Let $M(\la) \in \FF[\la]^{m\times (m+n)}$ be a polynomial matrix of full row rank $m$ and let $S_k$ be its Sylvester matrices for $k= 1,2, \ldots$. Let $r_k$ be the rank of $S_k$, $n_k$ be the right nullity of $S_k$, and $\alpha_k$ be the number of right minimal indices $\varepsilon_i$ of $M(\la)$ equal to $k$. Then these magnitudes obey the following recursive relations
\begin{equation} \label{recur}
\alpha_k = (n_{k+1}-n_k)-  (n_k-n_{k-1}) =  (r_k-r_{k-1})-(r_{k+1}-r_k), \quad k=1,2, \ldots ,
\end{equation}
initialized with $r_0=n_0=0$ and $\alpha_0=m+n-r_1=n_1$.
\end{theorem}

The second equality in \eqref{recur} between ranks and nullities easily follows from the identity
\begin{equation} \label{eq:rn}
r_k+n_k= k(m+n), \; \mbox{for}\; k\ge 1,
\end{equation}
and an intuitive proof of the first identity follows from the recursive definition of the null-spaces of $S_k$, for $k=1, 2, \ldots$. For $k=1$ the theorem says that the matrix
$$   S_1 := \left[ \begin{array}{c} M_0 \\ M_1 \\ \vdots \\ M_d \end{array}\right]  \in \FF^{(d+1)m \times (m+n)} $$
has a right null space of dimension $n_1$ equal to $\alpha_0$, the number of right minimal indices $\varepsilon_i$ of $M(\la)$ equal to $0$. Indeed, every
index $\varepsilon_i=0$ corresponds to a right null vector of degree $0$ (i.e.\/ constant) of $S_1$ and hence also of $M(\la)$.
For $k=2$, the matrix
$$   S_2 := \left[ \begin{array}{cc} M_0 & 0 \\ M_1 & M_0 \\ \vdots & M_1 \\ M_d & \vdots \\ 0 & M_d \end{array}\right] \, \in \FF^{(d+2)m \times 2(m+n)} $$
has a right null space of dimension $n_2$ equal to $\alpha_1+2\alpha_0$, because $S_2$ contains $S_1$ (padded with zeros) twice as a submatrix and hence the null-space of $S_1$ will contribute also twice to the null-space of $S_2$. The ``additional'' $\alpha_1$ linearly independent vectors in the null-space of $S_2$ correspond to the ``true" vectors of degree $1$ in $\mathcal{N}_r (M)$, i.e., those vectors in the null-space of $S_2$ which cannot be obtained as linear combinations of vectors in the null-space of $S_1$ padded with zeros. It then follows that $n_2-n_1 = \alpha_0+\alpha_1$. One uses the same arguments to show recursively that
\begin{equation} \label{sum} n_{k+1}-n_k = \sum_{i=0}^k \alpha_i \, , \end{equation}
which is equivalent to the first equality in \eqref{recur} together with the initializations $n_0 =0$, $n_1 = \alpha_0$. We refer to \cite{BKAK} and the references therein for a more detailed proof.
The following corollary, also given in \cite{BKAK}, is then easily derived from this.

\begin{corollary} \label{dprime} Let $M(\la) \in \FF[\la]^{m\times (m+n)}$ be a polynomial matrix of full row rank $m$ and let $S_k$ be its Sylvester matrices for $k =1,2,\ldots$. Let $r_k$ be the rank of $S_k$ and $n_k$ be the right nullity of $S_k$. If $d'$ is the smallest index $k$ for which
 \begin{equation}\label{eq:max}
    n_{k+1} - n_{k} = n, \quad \mathrm{or \;\; equivalently} \quad r_{k+1} - r_{k} = m,
 \end{equation}
 then $d'$ is the maximum right minimal index of $M(\la)$ or, equivalently, the maximum column degree of any minimal basis of $\mathcal{N}_r (M)$. Moreover, for all $k$ larger than $d'$, the equalities \eqref{eq:max} still hold.
 \end{corollary}

 \begin{proof}
 As soon as the $\alpha_i$, the number of right minimal indices of $M(\la)$ equal to $i$, add up to $n$, we have found a complete polynomial basis of the rational right null-space of $M(\la)$.
 There can be no further linearly independent right null vectors of $M(\la)$ or nonzero $\alpha_i$ since the sum $\sum_{i=0}^k \alpha_i =n$
is the total number of right minimal indices of $M(\la)$. The corresponding largest index $k$ with $\alpha_k\neq 0$ is therefore the largest right minimal index $d'$ of $M(\la)$. Since all $\alpha_k=0$ for $k$ larger than $d'$, equation \eqref{sum} guarantees that the equality \eqref{eq:max}
continues to hold. The equivalent condition $r_{k+1} - r_{k} = m$ follows from \eqref{eq:rn}.
 \end{proof}

We can also express the sum of the right minimal indices $\varepsilon_j$ of $M(\la)$ as a function of the nullities $n_k$.
\begin{corollary} \label{degreesum}
Let $M(\la) \in \FF[\la]^{m\times (m+n)}$ be a polynomial matrix of full row rank $m$ and let $S_k$ be its Sylvester matrices for $k=1,2,\ldots$. Let $n_k$ and $r_k$ be the right nullity and the rank of $S_k$, respectively, $\alpha_k$ be the number of right minimal indices $\varepsilon_i$ of $M(\la)$ equal to $k$, and $d'$ be the maximum right minimal index of $M(\la)$. Then
 \begin{equation}\label{dsum}
    \sum_{j=1}^n \varepsilon_j = \sum_{k=1}^{d'} k\alpha_k = \sum_{k=1}^{d'} k(n_{k-1}-2 n_k + n_{k+1})  = n \cdot d' - n_{d'}
= r_{d'} - m\cdot d'.
 \end{equation}
 \end{corollary}

\begin{proof}
The first identity follows from the definition of $\alpha_k$ which is the number of right minimal indices $\varepsilon_j$ equal to $k$. The second identity follows from
Theorem \ref{thm:ranktheorem}. The third identity follows from the relations of Theorem \ref{thm:ranktheorem}, which written as follows
\begin{equation} \label{matrixrelation}
 n_0=0, \quad \left[ \begin{array}{c} \alpha_0 \\\alpha_1 \\ \alpha_2 \\ \vdots \\ \alpha_{d'}  \end{array} \right]   =
 \left[ \begin{array}{cccccc} 1 \\ -2 & 1 \\ 1 & -2 & 1 \\ & \ddots & \ddots & \ddots \\  & & 1 & -2 & 1 \end{array} \right]
  \left[ \begin{array}{c} n_1 \\ n_2 \\ \vdots \\ n_{d'} \\  n_{d'+1} \end{array} \right],
\end{equation}
allow us to see immediately that $\sum_{k=1}^{d'} k \alpha_k = d' \cdot n_{d' +1} - (d'+1) n_{d'}$. This equality combined with the identity in Corollary \ref{dprime}, i.e.,
$$ n_{d'+1}= n_{d'} +n,
$$
gives $\sum_{k=1}^{d'} k \alpha_k = n \cdot d' -  n_{d'}$.
The fourth identity follows from \eqref{eq:rn} applied to $k=d'$.
 \end{proof}

 \begin{remark}
 Notice that equation \eqref{matrixrelation} in the proof of Corollary \ref{degreesum} establishes a one to one correspondence between the index sets $\{n_i,i=1,\ldots,d'+1\}$ and
  $\{\alpha_i,i=0,\ldots,d'\}$, as long as all the indices are non-negative and  $\sum_{i=0}^{d'}\alpha_i=n$ holds. This last identity is clearly equivalent to $n_{d'+1}-n_{d'}=n$.
 \end{remark}

We are now ready to formulate our necessary and sufficient finite number of rank conditions on constant matrices derived from $M(\la)$ to verify that it is a minimal basis. The conditions are given in the next theorem, where we note that $d' \leq \sum_{i=1}^m d_i$, as a consequence of Corollary \ref{dprime} and Theorem \ref{thm:dualbasis}, and so the number of rank conditions to be checked is indeed finite.

\begin{theorem} \label{th:finitenumbranks}
Let $M(\la) \in \FF[\la]^{m\times (m+n)}$ be a polynomial matrix, let $d_i, i=1,\ldots, m$, be its row degrees, and $M_{hr}$ be its highest-row-degree coefficient matrix. Let $S_k$ be the Sylvester matrices of $M(\la)$ for $k=1,2,\ldots$, and let $r_k$ and $n_k$ be the rank and the right nullity of $S_k$, respectively. Let $d'$ be the smallest index $k$ for which $n_{k+1}=n_k+n$, or equivalently, $r_{k+1}=r_k+m$. Then $M(\la)$ is a minimal basis if and only if the following rank conditions are satisfied
 \begin{equation}\label{eq:minimal}
    \rank M_{hr} = m \quad \mbox{and} \quad  r_{d'} - m \cdot d' = \sum_{i=1}^m d_i  \,.
 \end{equation}
 \end{theorem}

\begin{proof} Let us assume first that the conditions \eqref{eq:minimal} hold.
The condition $\rank M_{hr} = m$ says that $M(\la)$ is row reduced, which implies that $M(\la)$ has normal rank equal to its number of rows, $m$, (recall Remark \ref{rem:unicity}) and that it has a right null space of dimension $n$ and a set of $n$ right minimal indices $\varepsilon_j, j=1,\ldots, n$. These right minimal indices can then be computed via the recurrences of Theorem \ref{thm:ranktheorem} and condition $r_{d'} - m \cdot d' = \sum_{i=1}^m d_i$ implies (according to Corollaries \ref{degreesum} and \ref{dprime}) that
\begin{equation}  \label{minimality}
\sum_{i=1}^m d_i = \sum_{j=1}^n \varepsilon_j.
\end{equation}
Since $M(\la)$ has normal rank $m$, its rows form a polynomial basis of the row space of $M(\la)$ whose degree sum equals $\sum_{i=1}^m d_i$. If $\eta_1, \ldots , \eta_m$ are the minimal indices of the row space of $M(\la)$, then Theorem \ref{thm:dualbasis} combined with \eqref{minimality} imply
$\sum_{i=1}^m d_i = \sum_{j=1}^n \varepsilon_j = \sum_{j=1}^m \eta_j$, which in turn implies that the rows of $M(\la)$ must form a minimal basis, by definition of minimal basis.

The reverse implication follows immediately: assume that $M(\la)$ is a minimal basis. Then $\rank M_{hr} = m$ by Theorem \ref{minbasis_th}. In addition, Corollaries \ref{dprime} and \ref{degreesum} and Theorem \ref{thm:dualbasis} imply  $r_{d'} - m \cdot d' = \sum_{i=1}^m d_i$.
\end{proof}

Theorem \ref{th:finitenumbranks} is considerably simplified under the additional generic assumption that the leading coefficient of $M(\la)$ has full rank. Corollaries \ref{cor1:Mdfull} and \ref{cor2:Mdfull} are two results in that direction.

\begin{corollary} \label{cor1:Mdfull} Let $M(\la) = M_0 + M_1 \la + \cdots + M_d \la^d \in \FF[\la]^{m \times (m+n)}$ be a polynomial matrix such that $\rank M_d =m$. Let $S_k$ be the Sylvester matrices of $M(\la)$ for $k=1,2,\ldots$, let $r_k$ be the rank of $S_k$, and let $d'$ be the smallest index $k$ such that $r_{k+1} = r_k + m$. Then, $M(\la)$ is a minimal basis if and only if $r_{d'} = m (d+d')$.
\end{corollary}

\begin{proof} Note that the assumptions of Corollary \ref{cor1:Mdfull} imply that $M_{hr} = M_d$ and that the row degrees $d_1, d_2, \ldots, \allowbreak d_m$ of $M(\la)$ are all equal to $d$ and, so, $\sum_{i=1}^m d_i = m\cdot d$. In this scenario, if $M(\la)$ is a minimal basis, then it satisfies the second equality in \eqref{eq:minimal}, which implies $r_{d'} = m (d+d')$. Conversely, if one assumes that $r_{d'} = m (d+d')$ holds, then the second equality in \eqref{eq:minimal} is satisfied, while the first one is guaranteed by $\rank M_d = m$. Therefore $M(\la)$ is a minimal basis.
\end{proof}

\begin{corollary} \label{cor2:Mdfull}
Let $M(\la) = M_0 + M_1 \la + \cdots + M_d \la^d \in \FF[\la]^{m \times (m+n)}$ be a polynomial matrix and let $S_k$ be the Sylvester matrices of $M(\la)$ for $k=1,2,\ldots$. Then, $M(\la)$ is a minimal basis with $\rank M_d =m$ if and only if there exists an index $k$ such that $S_k$ has full row rank. In this case, if $d'$ is the smallest index $k$ for which $S_k$ has full row rank, then $d'$ is the largest right minimal index of $M(\la)$.
\end{corollary}

\begin{proof}
If $M(\la)$ is a minimal basis with $\rank M_d =m$, then Corollary \ref{cor1:Mdfull} implies $r_{d'} = m (d+d')$, which is the number of rows of $S_{d'}$. Then $S_{d'}$ has full row rank.

Conversely, if there exists an index $k$ such that $S_k$ has full row rank, then $M_d$ has full row rank, because the last block row of $S_k$ is $[0 \;|\; M_d]$, which also implies that $M(\la)$ has full row normal rank. Let $k_0$ be the smallest index $k$ such that $S_k$ has full row rank and denote by $r_k$ the rank of any Sylvester matrix $S_k$. Then, according to Lemma \ref{kl}, $S_{k_0 + 1}$ also has full row rank and their ranks satisfy
\begin{equation} \label{eq:innercor2}
r_{k_0 + 1} - r_{k_0} = m.
\end{equation}
However, $r_{k_0-1} < (k_0 -1 +d) m$, because $S_{k_0-1}$ has not full row rank. Therefore, $r_{k_0} - r_{k_0 - 1} > m$ and, so, $r_{k +1} - r_{k} > m$ for all $k \leq k_0 -1$, since Theorem \ref{thm:ranktheorem} implies $r_{j} - r_{j-1} \geq r_{j +1} - r_{j}$ for all $j\geq 1$ because $\alpha_j \geq 0$. Therefore, $k_0$ is the smallest index $k$ such that $r_{k + 1} = r_{k} + m$, that is, $k_0 = d'$ in Corollary \ref{cor1:Mdfull} and $r_{d'} = m\, (d' + d)$, since $S_{k_0} = S_{d'}$ has full row rank. So, Corollary \ref{cor1:Mdfull} implies that $M(\la)$ is a minimal basis. In addition, observe that also $k_0 = d'$ in Corollary \ref{dprime} and, so, $k_0 = d'$ is the largest right minimal index of $M(\la)$.
\end{proof}

We illustrate Theorems \ref{thm:ranktheorem} and \ref{th:finitenumbranks}, and Corollaries \ref{dprime}, \ref{cor1:Mdfull}, and \ref{cor2:Mdfull} with three simple examples. With the purpose that the reader can easily check that these results yield the right outcome, in our three examples Theorem \ref{minbasis_th} also allows us to see very easily whether $M(\la)$ is a minimal basis or not. In addition, we display a matrix $N(\la)$ such that the columns of $N(\la)^T$ are a minimal basis of $\mathcal{N}_r (M)$ (and, so, $N(\la)$ is a minimal basis dual to $M(\la)$ whenever $M(\la)$ is a minimal basis), which can again be easily checked by Theorem \ref{minbasis_th}.

\begin{example} \label{ex1} Let $M(\la)\in  \FF[\la]^{6\times 8}$ and $N(\la)\in  \FF[\la]^{2\times 8}$ be given:
$$M(\la) = \left[ \begin{array}{cccc} -I_2 & \la I_2 \\ & -I_2 & \la I_2  \\ & & -I_2 & \la I_2 \end{array} \right] ,
\quad N(\la) = \left[ \begin{array}{cccc} \la^3 I_2 & \la^2 I_2 & \la I_2  & I_2  \end{array} \right] \, .
$$
Then clearly $M(\la)N(\la)^T=0.$
The ranks $r_k$ and nullities $n_k$ of the Sylvester matrices  $S_k$ of $M(\la)$ are:
$$  r_1=8, r_2=16, r_3=24, r_4=30; \quad  n_1=0, n_2=0, n_3=0, n_4=2 \; .
$$
Then, clearly $d'=3$ and \eqref{recur} gives $[\alpha_0, \alpha_1, \alpha_2, \alpha_3] = [0,0,0,2]$ (which agrees with the row degrees of $N(\la)$). The condition \eqref{eq:minimal} becomes
$\rank M_{hr} = 6$ and $r_3-3\cdot m=24-3\cdot 6=6$, which is indeed equal to $d_1+d_2+d_3+d_4+d_5+d_6=6$. Therefore, $M(\la)$ is a minimal basis with right minimal indices $\{3,3\}$. Observe that for proving just that $M(\la)$ is a minimal basis (not for getting that its right minimal indices are $\{3,3\}$), one can use simply Corollary \ref{cor2:Mdfull} because the leading coefficient of $M(\la)$ has clearly full rank. In this case the numbers of rows of $S_k$ for $k=1,2,3$ are $12,18,24$, respectively, which implies that $S_3$ has full row rank and that the largest right minimal index of $M(\la)$ is $3$.
\end{example}

The second example corresponds to a polynomial matrix that is not a minimal basis and has three different right minimal indices.
\begin{example} Let $M(\la)\in  \FF[\la]^{4\times 7}$ and $N(\la)\in  \FF[\la]^{3\times 7}$ be given:
$$M(\la) = \left[ \begin{array}{cc|cc|cccc} \la & 0 & & \\ \hline & & -1 & \la \\ \hline & & & & -1 & \la \\ & & & & & -1 & \la \end{array} \right] ,
\quad N(\la) = \left[ \begin{array}{cc|cc|ccc} 0 & 1 & & \\ \hline & & \la & 1 \\ \hline & & & & \la^2 & \la & 1 \end{array} \right] .
$$
Then clearly $M(\la)N(\la)^T=0$. The ranks $r_k$ and nullities $n_k$ of the Sylvester matrices  $S_k$ of $M(\la)$ are:
$$  r_1=6, r_2=11, r_3=15; \quad  n_1=1, n_2=3, n_3=6 \, .
$$
Then clearly $d'=2$ and \eqref{recur} gives $[\alpha_0,\alpha_1,\alpha_2]= [1,1,1]$ (which agrees with the row degrees of $N(\la)$). The condition \eqref{eq:minimal} becomes
$\rank M_{hr} = 4$ and $r_2-2\cdot m=11-2\cdot 4=3$, which is not equal to $d_1+d_2+d_3+d_4=1+1+1+1=4$. Therefore $M(\la)$ is not a minimal basis. Observe that for proving just that $M(\la)$ is not a minimal basis (not for getting that its right minimal indices are $\{0,1,2\}$), one can use simply Corollary \ref{cor1:Mdfull} because the leading coefficient of $M(\la)$ has clearly full rank.
\end{example}

In the previous two examples $M(\la)$ has degree $1$ and all its row degrees equal. The polynomial matrix $M(\la)$ in the next example does not satisfy any of these two properties.

\begin{example} \label{ex3} Let $M(\la)\in  \FF[\la]^{6\times 8}$ and $N(\la)\in  \FF[\la]^{2\times 8}$ be given:
$$M(\la) = \left[ \begin{array}{cccc} -I_2 & \la I_2 \\ & -I_2 & \la I_2  \\ & & -I_2 & \la^2 I_2 \end{array} \right] ,
\quad N(\la) = \left[ \begin{array}{cccc} \la^4 I_2 & \la^3 I_2 & \la^2 I_2  & I_2  \end{array} \right] \, .
$$
Then clearly $M(\la)N(\la)^T=0.$
The ranks $r_k$ and nullities $n_k$ of the Sylvester matrices  $S_k$ of $M(\la)$ are:
$$
r_1 = 8, r_2 = 16, r_3 = 24, r_4 = 32, r_5 = 38; \quad
n_1 = 0, n_2 = 0, n_3 = 0, n_4 = 0, n_5 = 2 \, .
$$
Then clearly $d'=4$ and \eqref{recur} gives $[\alpha_0, \alpha_1, \alpha_2, \alpha_3, \alpha_4] =[0,0,0,0,2]$ (which agrees with the row degrees of $N(\la)$). The condition \eqref{eq:minimal} becomes $\rank M_{hr} = 6$ and $r_4 - 4\cdot m = 32 - 4\cdot 6 = 8$, which is equal to $d_1 + d_2 + d_3 + d_4 + d_5 + d_6 = 8$. Therefore, $M(\la)$ is a minimal basis. In this case the leading coefficient of $M(\la)$ (the one corresponding to degree $2$) has not full row rank and Corollaries \ref{cor1:Mdfull} and \ref{cor2:Mdfull} cannot be used.
\end{example}

We emphasize that in the examples above the rank conditions are completely in terms of the coefficient matrices of $M(\la)$ via its Sylvester matrices, and that the minimal bases $N(\la)$ are displayed only for the purpose of comparison.

\section{Full-Sylvester-rank polynomial matrices and their properties}
\label{sect:RankProp} In this section we characterize the polynomial matrices of size $m \times (m+n)$ and degree at most $d$ all of whose Sylvester matrices $S_k$ defined in \eqref{eq:Sylvester} have full rank, either full column rank when $S_k$ has more rows than columns or full row rank otherwise. We advance that such matrices are always minimal bases and that satisfy other additional properties.
The {\em ceiling} function of a real number $x$ is often used in the rest of this paper and is denoted by $\lceil x \rceil$. Recall that $\lceil x \rceil$ is the smallest integer that is larger than or equal to $x$.

The following definition will allow us to refer to the property of interest in this section in a concise way.
\begin{definition} \label{def.fullsylrank} Let $M(\la) \in \FF[\la]^{m \times (m+n)}$ be a polynomial matrix of degree at most $d$, let $S_k$ for $k=1,2,\ldots$ be its Sylvester matrices, and let $r_k$ be the rank of $S_k$. The polynomial matrix $M(\la)$ is said to have full-Sylvester-rank if all the matrices $S_k$ have full rank, i.e., if $r_k = \min\{(k+d)m \, , \, k(m+n)\}$ for $k=1,2,\ldots$.
\end{definition}

It is necessary and sufficient to check at most two ranks for determining whether a polynomial matrix has full-Sylvester-rank or not, as a consequence of Lemmas \ref{lk} and \ref{kl}. This is stated in Lemma \ref{lemm.two-ranks}.

\begin{lemma} \label{lemm.two-ranks}Let $M(\la) \in \FF[\la]^{m \times (m+n)}$ be a polynomial matrix of degree at most $d$, let $S_k$ for $k=1,2,\ldots$ be its Sylvester matrices, and let
\begin{equation} \label{eq.ktprime}
k' := \left\lceil \frac{m d}{n} \right\rceil \quad \mbox{and} \quad n k' = m d + t, \quad \mbox{where $0\leq t < n$.}
\end{equation}
Then the following statements hold.
\begin{itemize}
\item[\rm (a)] $k'$ is the smallest index $k$ for which the number of columns of $S_k$ is larger than or equal to the number of rows of $S_k$.
\item[\rm (b)] If $k' > 1$ and $t>0$, then $M(\la)$ has full-Sylvester-rank if and only if $S_{k'-1}$ has full column rank and $S_{k'}$ has full row rank.
\item[\rm (c)] If $k' = 1$ or $t=0$, then $M(\la)$ has full-Sylvester-rank if and only if $S_{k'}$ has full row rank.
\end{itemize}
\end{lemma}

\begin{proof}
Part (a) follows from the size of $S_{k}$ displayed in \eqref{eq:Sylvester}, because $(k+d) m \leq k(m+n)$ is equivalent to $d m/n \leq k$, and this is equivalent to $k' \leq k$. Part (b) follows from Lemmas \ref{lk} and \ref{kl}, together with the fact that if $t>0$ then $S_{k'}$ has strictly more columns than rows, while $S_{k'-1}$ has strictly less columns than rows. Note that $S_{k'-1}$ is defined since $k'>1$. In part (c) we have two scenarios. If $t=0$, then $S_{k'}$ is a square matrix and the result follows again from Lemmas \ref{lk} and \ref{kl} for any value of $k'$. If $k'=1$ and $t>0$, the same argument as in part (b) proves the result, with the only difference that in this case $S_{k'-1}$ is not defined.
\end{proof}

Note that Corollary \ref{cor2:Mdfull} immediately implies that polynomial matrices with full-Sylvester-rank are minimal bases whose leading coefficient has full rank. So, they have full row normal rank and, since the ranks $r_k$ of their Sylvester matrices are given by $r_k = \min\{(k+d)m \, , \, k(m+n)\}$, their right minimal indices are fixed by the recurrence in Theorem \ref{thm:ranktheorem}. This leads to a characterization of full-Sylvester-rank matrices in terms of their {\em complete eigenstructure}, i.e., their finite and infinite elementary divisors\footnote{In this paper the infinite elementary divisors of a polynomial matrix $M(\la)$ with degree at most $d$ are the elementary divisors associated to the eigenvalue $0$ of the {\em reversal polynomial matrix} $\mbox{rev}_d M(\la) := \la^d M(1/\la)$ \cite{DDM}.} and their left and right minimal indices. This characterization together with other properties are presented in Theorem \ref{thm.propsfullSylvrank}.

\begin{theorem} \label{thm.propsfullSylvrank} Let $M(\la) = M_0 + M_1 \la + \cdots + M_d \la^d \in \FF[\la]^{m \times (m+n)}$ be a polynomial matrix of degree at most $d$, let $\alpha_k$ be the number of right minimal indices of $M(\la)$ equal to $k$, and let $k'$ and $t$ be defined as in \eqref{eq.ktprime}. Then the following statements hold.
\begin{itemize}
\item[\rm (a)] $M(\la)$ has full-Sylvester-rank if and only if the complete eigenstructure of $M(\la)$ consists only of the following right minimal indices
\begin{equation} \label{eq.statfullSylrank}
\alpha_{k'-1} = t, \quad \alpha_{k'} = n-t, \quad \mbox{and} \quad  \alpha_j = 0 \; \; \mbox{for $j\notin \{k'-1,k'\}$}.
\end{equation}
\item[\rm (b)] If $M(\la)$ has full-Sylvester-rank, then $M(\la)$ is a minimal basis with $\rank M_d = m$, i.e., with all its row degrees equal to $d$.

\item[\rm (c)] If $M(\la)$ has full-Sylvester-rank, then the degree of any minimal basis dual to $M(\la)$ is equal to $k'$. That is, with the notation of Corollaries \ref{dprime} and \ref{degreesum}, $d' = k'$ holds for full-Sylvester-rank matrices.
\end{itemize}

\end{theorem}

\begin{proof}
Proof of (a). We assume first that $M(\la)$ has full-Sylvester-rank. Then for all $k\geq k'$ the Sylvester matrix $S_k$ of $M(\la)$ has full row rank, since $S_k$ has full rank and does not have more rows than columns. Therefore, Corollary \ref{cor2:Mdfull} implies that $M(\la)$ is a minimal basis and that $M_d$ has full row rank. These properties imply in turn that the complete eigenstructure of $M(\la)$ consists only of $n$ right minimal indices, since $M(\la)$ has not finite elementary divisors as a consequence of Theorem \ref{minbasis_th}, $M(\la)$ has not infinite elementary divisors as a consequence of $\rank M_d = m$, $M(\la)$ has not left minimal indices because has full row normal rank, and the number of right minimal indices of $M(\la)$ is $\dim \mathcal{N}_r (M) = m+n - \rank (M) = n$. It remains to determine the $n$ right minimal indices of $M(\la)$.
To this purpose note that the full-Sylvester-rank property together with Lemma \ref{lemm.two-ranks}-(a) imply that $S_k$ has full column rank for all $1 \leq k < k'$ and has full row rank for all $k \geq k'$, which in terms of the rank, $r_k$, and right nullity, $n_k$, of $S_k$ is equivalent to
\begin{equation} \label{eq.fullSylvranknull}
\begin{array}{lll}
r_k = k (m+n), \quad &n_k = 0, \quad & \mbox{for $1 \leq k < k'$, \; and} \\
r_k = (k+d) m, \quad &n_k = nk - dm, \quad & \mbox{for $k \geq k'$.}
\end{array}
\end{equation}
The relations \eqref{eq.fullSylvranknull} can be expressed also in terms of $k'$ and $t$ in \eqref{eq.ktprime} as follows
\begin{equation} \label{eq.fullSylvranknull2}
\begin{array}{lll}
r_k = k (m+n), \quad &n_k = 0, \quad & \mbox{for $1 \leq k < k'$, \; and} \\
r_k = (k+d) m, \quad &n_k = n(k-k') + t, \quad & \mbox{for $k \geq k'$.}
\end{array}
\end{equation}
Finally, from Theorem \ref{thm:ranktheorem} we have that $\alpha_0 = n_1$ and $\alpha_k = n_{k+1} - 2 n_k + n_{k-1}$ for $k \geq 1$, which combined with \eqref{eq.fullSylvranknull2} yields
\begin{equation} \label{eq.fullSylvranknull3}
\begin{array}{lll}
\alpha_k = 0, \quad & \mbox{for $1 \leq k \leq k'-2$,} \\
\alpha_{k' -1} = n_{k'} = t,\\
\alpha_{k'} = n_{k'+1} - 2 n_{k'} = n-t,\\
\alpha_{k} = 0, \quad & \mbox{for $k > k'$,}
\end{array}
\end{equation}
where the last line follows immediately from the fact that the $n$ right minimal indices of $M(\la)$ have been already determined in the previous lines of \eqref{eq.fullSylvranknull3}.
Observe that according to \eqref{eq.ktprime} $k'\geq 1$ and that in the limit case $k' = 1$, i.e., $md \leq n$, we get from \eqref{eq.fullSylvranknull3} $\alpha_{k' -1} = \alpha_0 = n_{1} = t$, which is consistent with the initialization in Theorem \ref{thm:ranktheorem}. In this limit case the first lines in \eqref{eq.fullSylvranknull}, \eqref{eq.fullSylvranknull2}, and \eqref{eq.fullSylvranknull3} are not present. The following table illustrates the nullities, ranks, and $\alpha_k$ numbers for $k = k'-2 , k'-1 , k' , k'+1 , k'+2$:
$$ \begin{array}{c|ccccc} k & k'-2 & k'-1 & k' & k'+1 & k'+2 \\ \hline n_k & 0 & 0 & t & t+n & t+2n \\
r_k & (k'-2)(m+n) & (k'-1)(m+n) & (k'+d)m & (k'+d+1)m & (k'+d+2)m \\
\alpha_k & 0 & t & n-t & 0 & 0  \end{array} $$

Next we prove the sufficiency in part (a). Assume that the complete eigenstructure of $M(\la)$ consists only of the $n$ right minimal indices described in \eqref{eq.statfullSylrank}. Since the number of right minimal indices is precisely $n$, $M(\la)$ has full normal rank, i.e., $\rank (M) = m$. Then, we can apply Theorem \ref{thm:ranktheorem} to $M(\la)$. Note that \eqref{recur} implies \eqref{sum} and, so, the right nullities $n_k$ of the Sylvester matrices are determined from the $\alpha_k$ numbers as follows
\begin{equation} \label{eq.fullSylvranknull4}
n_1 = \alpha_0, \qquad n_{k+1} = n_k + \sum_{i=0}^k \alpha_i \quad \mbox{for $k \geq 1$}.
\end{equation}
This recursion combined with \eqref{eq.statfullSylrank} yields
\[
\begin{array}{ll}
n_k = 0, \quad &\mbox{for $k < k'$},\\
n_k = n (k-k') + t, \quad &\mbox{for $k \geq k'$},
\end{array}
\]
which according to \eqref{eq.fullSylvranknull2} implies that every Sylvester matrix $S_k$ of $M(\la)$ has full rank. This completes the proof of part (a).

Part (b) has been already proved at the beginning of the proof of part (a). Part (c) follows from part (a) and the fact that the row degrees of any minimal basis dual to $M(\la)$ are precisely the right minimal indices of $M(\la)$.
\end{proof}

Theorem \ref{thm.propsfullSylvrank}-(a) allows us to state in Theorem \ref{thm.2necsuffullSylvRank} another necessary and sufficient condition for a polynomial matrix to have full-Sylvester-rank in terms only of its right minimal indices or, equivalently, the degrees of their dual minimal bases.

\begin{theorem} \label{thm.2necsuffullSylvRank} Let $M(\la) = M_0 + M_1 \la + \cdots + M_d \la^d \in \FF[\la]^{m \times (m+n)}$ be a polynomial matrix of degree at most $d$, let $\alpha_k$ be the number of right minimal indices of $M(\la)$ equal to $k$, and let $k'$ and $t$ be defined as in \eqref{eq.ktprime}. Then, $M(\la)$ has full-Sylvester-rank if and only if the right minimal indices of $M(\la)$ are
\begin{equation} \label{eq.2statfullSylrank}
\alpha_{k'-1} = t, \quad \alpha_{k'} = n-t, \quad \mbox{and} \quad  \alpha_j = 0 \; \; \mbox{for $j\notin \{k'-1,k'\}$}.
\end{equation}
\end{theorem}

\begin{proof}
Theorem \ref{thm.propsfullSylvrank}-(a) implies that if $M(\la)$ has full-Sylvester-rank, then its right minimal indices are those in \eqref{eq.2statfullSylrank}, which proves the necessity. The sufficiency is proved as follows. If the right minimal indices of $M(\la)$ are those in \eqref{eq.2statfullSylrank}, then $\dim \mathcal{N}_r (M) = n$ and $\rank (M) = m$. Therefore, $M(\la)$ has not left minimal indices. In addition, the Index Sum Theorem \cite[Theorem 6.5]{DDM} applied to the polynomial matrix $M(\la)$ of grade $d$ implies that $M(\la)$ has neither finite nor infinite elementary divisors, since the sum of the right minimal indices of $M(\la)$ is
\[
(k' - 1) \, \alpha_{k' -1} + k' \, \alpha_{k'} = (k' - 1) \, t + k' \, (n-t) = -t + k' \, n = m d = \rank(M) \, \mbox{grade} (M) \, .
\]
Therefore, the complete eigenstructure of $M(\la)$ consists only of the right minimal indices in \eqref{eq.2statfullSylrank} and Theorem \ref{thm.propsfullSylvrank}-(a) implies that $M(\la)$ has full-Sylvester-rank.
\end{proof}

\begin{remark} \label{rem.otherminbases}
Observe that Theorem \ref{thm.propsfullSylvrank}-(b) states that full-Sylvester-rank matrices are minimal bases with full rank leading matrix coefficient $M_d$. We emphasize that the converse result is not true: a minimal basis $C(\la) \in \FF[\la]^{m \times (m+n)}$ with degree at most $d$ and leading coefficient $C_d$ of full rank has not necessarily full-Sylvester-rank. This follows immediately from Theorem \ref{thm:dualbasis} because there exist dual minimal bases $C(\la) \in \FF[\la]^{m \times (m+n)}$ and $D(\la) \in \FF[\la]^{n \times (m+n)}$ with the row degrees of $C(\la)$ all equal to $d$ (equivalently $C_d$ has full rank) and with the row degrees of $D(\la)$ having arbitrary values whose sum is $md$. Since the row degrees of $D(\la)$ are the right minimal indices of $C(\la)$, they can be different than those in \eqref{eq.2statfullSylrank} and so $C(\la)$ has not full-Sylvester-Rank.
\end{remark}

\section{Genericity of full-Sylvester-rank matrices and consequences} \label{sec.genericity}
It is well-known that $p\times q$ constant matrices have {\em generically}, i.e., typically or ``almost always'' if the entries are considered as random variables, full rank equal to $\min\{p,q\}$. Therefore, it is natural to expect that {\em generically} all of the Sylvester matrices of a polynomial matrix $M(\la)\in \FF[\la]^{m \times (m+n)}$ have full rank. In other words, it is natural to expect that {\em generically} a polynomial matrix $M(\la)\in \FF[\la]^{m \times (m+n)}$ has full-Sylvester-rank. However, this has to be rigorously proved, since Sylvester matrices are highly structured matrices containing many zero entries and with block columns intimately related each other. The development of such rigorous proof and the analysis of some interesting consequences of this result are the goals of this section. To this purpose, we need to define the precise meaning of {\em genericity}, which in this work is essentially the standard notion in Algebraic Geometry.

We define genericity inside the vector space $\FF[\la]^{m \times (m+n)}_d$ of polynomial matrices of size $m\times (m+n)$ and degree at most $d$, where in this section, and in the rest of the paper, $\FF = \RR$ or $\FF = \CC$. One motivation for considering $\FF[\la]^{m \times (m+n)}_d$ as our ``ambient'' space comes from the applications to backward error analyses of polynomial eigenproblems solved via linearizations or $\ell$-ifications (see \cite{DDM,DDV-l-ifications} for the definition of $\ell$-ification) that we have in mind for the results in this paper. In practice, backward errors are considered arbitrary perturbations, since the only information available on them is their magnitude, which do not increase the degree of the polynomial matrix \cite{blockKron,tissuer-back-err}.

The first step in the definition of genericity is to identify $\FF[\la]^{m \times (m+n)}_d$ with $\RR^{(d+1)m(m+n)}$ when $\FF = \RR$, or
with $\RR^{2\,(d+1)m(m+n)}$ when $\FF = \CC$. If $\FF = \RR$ such identification can be made, for instance, by mapping each polynomial matrix
$M(\la) = M_0 + M_1 \la + \cdots + M_d \la^d \in \FF[\la]^{m \times (m+n)}_d$ into
$\mbox{vec} ([M_0 \; M_1 \; \cdots \; M_d ]) \in \RR^{(d+1)m(m+n)}$, where $\mbox{vec} (\cdot)$ is the standard vectorization operator defined for instance in \cite[Chapter 4]{HoJo}. If $\FF = \CC$, one considers the entrywise real and imaginary parts of each matrix coefficient $M_i$, denoted by $\mbox{Re} (M_i)$ and $\mbox{Im} (M_i)$, respectively, and the identification is made by mapping $M(\la)$ into $\mbox{vec} ([\mbox{Re}(M_0) \; \mbox{Im} (M_0) \; \cdots \allowbreak \; \mbox{Re}(M_d) \; \mbox{Im} (M_d)]) \in \RR^{2(d+1)m(m+n)}$. Next, we recall that an {\em algebraic set} in $\RR^p$ is the set of common zeros of a finite number of multivariable polynomials with $p$ variables and coefficients in $\RR$, and that an algebraic set is {\em proper} if it is not the whole set $\RR^p$. With these concepts at hand, the standard definition of genericity of Algebraic Geometry is as follows: a {\em generic set of $\RR^p$ is a subset of $\RR^p$ whose complement is contained in a proper algebraic set}. This definition extends obviously to the corresponding one of {\em generic set of $\FF[\la]^{m \times (m+n)}_d$}, with $\FF = \RR$ or $\FF = \CC$, through the bijective ``$\mbox{vec}$'' mappings discussed above, since these mappings allow us to define algebraic sets of $\FF[\la]^{m \times (m+n)}_d$. In the sequel, expressions as ``generically the polynomial matrices in $\FF[\la]^{m \times (m+n)}_d$ have the property $\mathcal{P}$'' have the precise meaning of ``the polynomial matrices of $\FF[\la]^{m \times (m+n)}_d$ that satisfy property $\mathcal{P}$ are a  generic set of $\FF[\la]^{m \times (m+n)}_d$''.

Now, we are in the position of stating and proving the main result of this section.

\begin{theorem} \label{thm.gensylvfullrank} Let $\mathrm{Syl}[\la]^{m \times (m+n)}_d \subset \FF[\la]^{m \times (m+n)}_d$ be the set of polynomial matrices of size $m\times (m+n)$, degree at most $d$, and with full-Sylvester-rank. Then the complement of $\mathrm{Syl}[\la]^{m \times (m+n)}_d$ is a proper algebraic set of $\FF[\la]^{m \times (m+n)}_d$ and, so, $\mathrm{Syl}[\la]^{m \times (m+n)}_d$
is a generic set of $\FF[\la]^{m \times (m+n)}_d$.
\end{theorem}

\begin{proof} Let $k'$ and $t$ be defined as in \eqref{eq.ktprime}. We will prove the theorem in the case $k'>1$ and $t>0$. The proofs in the cases $k' = 1$ or $t=0$ are similar and are omitted. Taking into account Lemma \ref{lemm.two-ranks}-(b), we have that
$$
\mathrm{Syl}[\la]^{m \times (m+n)}_d = \{M(\la) \in  \FF[\la]^{m \times (m+n)}_d \, :\, \det(S_{k'-1}^* S_{k'-1}) \cdot \det(S_{k'} S_{k'}^*) \ne 0\},
$$
where $S_{k'-1}$ and $S_{k'}$ are Sylvester matrices of $M(\la)$. The complement of $\mathrm{Syl}[\la]^{m \times (m+n)}_d$ relative to $\FF[\la]^{m \times (m+n)}_d$ is
$$
\left( \mathrm{Syl}[\la]^{m \times (m+n)}_d \right)^c = \{M(\la) \in  \FF[\la]^{m \times (m+n)}_d \, :\, \det(S_{k'-1}^* S_{k'-1}) \cdot \det(S_{k'} S_{k'}^*) = 0\},
$$
which is obviously an algebraic set. More precisely: if $\FF = \RR$, $\left( \mathrm{Syl}[\la]^{m \times (m+n)}_d \right)^c$ is the set of zeros of one multivariable polynomial in the entries of the matrix coefficients $M_i$, $i=0,1,\ldots, d$, of $M(\la)$, and if $\FF = \CC$, the real and imaginary parts of $\det(S_{k'-1}^* S_{k'-1}) \cdot \det(S_{k'} S_{k'}^*) = 0$ are equivalent to two multivariable polynomial equations in the real and imaginary parts of the entries of the matrix coefficients $M_i$, $i=0,1,\ldots, d$, of $M(\la)$. It only remains to prove that $\left( \mathrm{Syl}[\la]^{m \times (m+n)}_d \right)^c$ is proper, i.e., that there is at least one polynomial matrix $M(\la) \in \FF[\la]^{m \times (m+n)}_d$ such that $M(\la) \notin \left( \mathrm{Syl}[\la]^{m \times (m+n)}_d \right)^c$. This follows immediately from Theorem \ref{thm:dualbasis}, which guarantees the existence of dual minimal bases $M(\la) \in  \FF[\la]^{m \times (m+n)}$ and
$N(\la) \in  \FF[\la]^{n \times (m+n)}$ with all the row degrees of $M(\la)$ equal to $d$ (so $M(\la) \in  \FF[\la]^{m \times (m+n)}_d$ ), and with $t$ row degrees of $N(\la)$ equal to $k'-1$ and the other $n-t$ equal to $k'$, because the row degrees of each of these two matrices sum up $md$. Therefore, the right minimal indices of $M(\la)$, which are the row degrees on $N(\la)$, are the ones in \eqref{eq.2statfullSylrank} and, by Theorem \ref{thm.2necsuffullSylvRank}, $M(\la) \in  \mathrm{Syl}[\la]^{m \times (m+n)}_d $ and $M(\la) \notin \left( \mathrm{Syl}[\la]^{m \times (m+n)}_d \right)^c$.
\end{proof}

We have just proved that generically the polynomial matrices in $\FF[\la]^{m \times (m+n)}_d$ have full-Sylvester-rank. This can be combined with Theorem \ref{thm.propsfullSylvrank} to prove that generically the polynomial matrices in $\FF[\la]^{m \times (m+n)}_d$ satisfy other interesting properties, in particular, the property of being minimal bases whose degree-$d$ matrix coefficient has full row rank. These properties are stated in the following corollary, whose simple proof is omitted.

\begin{corollary} \label{cor:genproperties} Let $k'$ and $t$ be defined as in \eqref{eq.ktprime}. Then, the following subsets of $\FF[\la]^{m \times (m+n)}_d$ are generic in $\FF[\la]^{m \times (m+n)}_d$:
\begin{itemize}
\item[\rm (a)] The set of $m \times (m+n)$ polynomial matrices of degree at most $d$ and whose complete eigenstructure consists only of the following right minimal indices
\[
\alpha_{k'-1} = t, \quad \alpha_{k'} = n-t, \quad \mbox{and} \quad  \alpha_j = 0 \; \; \mbox{for $j\notin \{k'-1,k'\}$},
\]
where $\alpha_j$ denotes the number of right minimal indices equal to $j$.
\item[\rm (b)]  The set of $m \times (m+n)$ polynomial matrices of degree at most $d$ that are minimal bases with degree-$d$ matrix coefficient of full row rank.

\item[\rm (c)] The set of $m \times (m+n)$ polynomial matrices of degree at most $d$ that are minimal bases and such that their dual minimal bases have degree equal to $k'$.
\end{itemize}

\end{corollary}

Observe that the set defined in Corollary \ref{cor:genproperties}-(a) is precisely $\mathrm{Syl}[\la]^{m \times (m+n)}_d$, while $\mathrm{Syl}[\la]^{m \times (m+n)}_d$ is strictly included in the sets $\mathcal{A}$ and $\mathcal{B}$ defined in parts (b) and (c) of Corollary \ref{cor:genproperties}, respectively, as a consequence of the discussion in Remark \ref{rem.otherminbases}. Therefore the complements of these sets satisfy $\mathcal{A}^c \subset \left( \mathrm{Syl}[\la]^{m \times (m+n)}_d \right)^c$ and $\mathcal{B}^c \subset \left( \mathrm{Syl}[\la]^{m \times (m+n)}_d \right)^c$, respectively.

\section{Robustness of minimal bases and of full-Sylvester-rank matrices} \label{sec.smoothness}
This section has three goals: first, to characterize when a minimal basis $M(\la) \in \FF[\la]^{m \times (m+n)}_d$ is robust under perturbations in the sense that all the polynomial matrices in a neighborhood of $M(\la)$ are also minimal bases; second, to estimate the size of such neighborhood; and, third, to prove that full-Sylvester-rank matrices are robust and to estimate the sizes of the corresponding neighborhoods of robustness.

Throughout the rest of the paper the singular values of a constant matrix $A \in \FF^{p \times q}$ are denoted by $\sigma_{1} (A) \geq \sigma_{2} (A) \geq \cdots \geq \sigma_{\min\{p,q\}} (A)$ and the Sylvester matrices of any $P(\la) \in \FF[\la]^{m \times (m+n)}_d$ are denoted by $S_k (P)$ for $k =1, 2,\ldots$. In order to study the questions of interest in this section, we define a norm in $\FF[\la]^{m \times (m+n)}_d$ as follows: the norm of any $P(\la) \in \FF[\la]^{m \times (m+n)}_d$ is $\|S_1 (P)\|_2$,
where $\|A\|_2 = \sigma_{1} (A)$ is the standard spectral norm of the matrix $A$ \cite{stewart-sun}. This norm induces the distance $\|S_1 (P)- S_1 (\widetilde{P})\|_2 = \|S_1 (P - \widetilde{P})\|_2$ between any two polynomial matrices $P(\la), \widetilde{P} (\la) \in \FF[\la]^{m \times (m+n)}_d$.

In this section, as well as in the rest of this manuscript, we will use very often Lemma \ref{C}, whose proof relies on Lemma \ref{A}.
\begin{lemma} \label{A}
Let $A=\left[\begin{array}{c|c|c|c} A_1 & A_2 & \cdots & A_k \end{array} \right]$ where $A_i\in \FF^{m\times n_i}, i=1,\ldots,k$.
Then $$ \max_i \sigma_{1}(A_i) \le \sigma_{1}(A) \le \sqrt{\sum_{i=1}^k \sigma_{1}^2(A_i)} \le \sqrt{k} \cdot \max_i \sigma_{1}(A_i).$$
\end{lemma}
\begin{proof}
This is a simple consequence of \cite[Corollary 3.1.3]{HoJo} and \cite[Problem 22 in p. 217]{HoJo}.
\end{proof}

\begin{lemma} \label{C} Let $P(\la) \in \FF[\la]^{m \times (m+n)}_d$. Then the following inequalities hold for the Sylvester matrices of $P(\la)$:
$$ \| S_1 (P)\|_2 \le \|S_{k} (P)\|_2 \le \sqrt{k} \cdot \| S_1 (P)\|_2.  $$
\end{lemma}
\begin{proof}
This is a direct consequence of Lemma \ref{A} and the structure of $S_{k} (P)$.
\end{proof}

Next theorem proves that a minimal basis of degree at most $d$ is robust inside $\FF[\la]^{m \times (m+n)}_d$ if and only if its row degrees are all equal to $d$ or, equivalently, if and only if its degree $d$ matrix coefficient has full rank.

\begin{theorem} \label{thm.smoothminbases} Let $M(\la) = M_0 + M_1 \la + \cdots + M_d \la^d \in \FF[\la]^{m \times (m+n)}_d$ be a minimal basis. Then the following statements hold:
\begin{itemize}
\item[\rm (a)] If $\rank M_d < m$, then for all $\epsilon >0$ there exists a polynomial matrix $\widetilde{M}(\la)\in \FF[\la]^{m \times (m+n)}_d$ that is not a minimal basis and satisfies $\|S_1 (M)- S_1 (\widetilde{M})\|_2 < \epsilon$. That is, as close as we want to $M(\la)$ there are polynomial matrices that are not minimal bases.

\item[\rm (b)] If $\rank M_d = m$, then there exists an index $k$ such that $S_k (M)$ has full row rank and every polynomial matrix $\widetilde{M}(\la) = \widetilde{M}_0 + \widetilde{M}_1 \la + \cdots + \widetilde{M}_d \la^d \in \FF[\la]^{m \times (m+n)}_d$ that satisfies
    \begin{equation} \label{eq.smoothneigh1}
    \|S_1 (M)- S_1 (\widetilde{M})\|_2 < \frac{\sigma_{(k+d)m} (S_{k}(M))}{\sqrt{k}}
    \end{equation}
    is a minimal basis with $\rank \widetilde{M}_d =m$.  That is, all the polynomial matrices sufficiently close to $M(\la)$ are minimal bases with full rank leading coefficient.
\end{itemize}
\end{theorem}

\begin{proof}
Proof of (a). Note that $M_d$ has at least one zero row, because otherwise all the row degrees of $M(\la)$ would be equal to $d$, then $M_d$ would be the highest-row-degree coefficient matrix of $M (\la)$ (recall Definition \ref{colred}), and $\rank M_d = m$ by Theorem \ref{minbasis_th}, which is a contradiction.

Assume first $m > 1$, then either $M_d \ne 0$ or $M_d = 0$ (in this latter case the degree of $M(\la)$ is strictly less than $d$). In the case $M_d \ne 0$, $M_d$ has at least one zero row and at least one nonzero row $w_k \ne 0$. A polynomial matrix $\widetilde{M} (\la)$ as in the statement can be constructed as follows: $\widetilde{M} (\la)$ is equal to $M(\la)$ except that one of the zero rows of $M_d$ is replaced by the row vector $(0.5 \epsilon/\|w_k\|_2) \, w_k$. This implies that the highest-row-degree coefficient matrix $\widetilde{M}_{hr}$ of $\widetilde{M} (\la)$ has two linearly dependent rows, so $\widetilde{M} (\la)$ is not a minimal basis by Theorem \ref{minbasis_th}, and that $\|S_1 (M)- S_1 (\widetilde{M})\|_2 = 0.5 \epsilon < \epsilon$. In the case $M_d = 0$, a polynomial $\widetilde{M} (\la)$ as in the statement can be constructed as follows: $\widetilde{M} (\la)$ is equal to $M(\la)$ except that two zero rows of $M_d$ are both replaced by the same arbitrary vector $v_\epsilon$ with norm $\|v_\epsilon \|_2 < \epsilon /2$. Then, $\widetilde{M}_{hr}$ of $\widetilde{M} (\la)$ has two rows equal to $v_\epsilon$, so $\widetilde{M} (\la)$ is not a minimal bases by Theorem \ref{minbasis_th}, and $\|S_1 (M)- S_1 (\widetilde{M})\|_2 = \|[v_\epsilon \, ; \, v_\epsilon]\|_2 < \epsilon$.

The limiting case $m=1$, i.e., when $M(\la)$ has only one row and has degree smaller than $d$, requires a different proof. Note that in this case $\rank M_d < m =1$ is equivalent to $M_d =0$. Let us express $M(\la) = [m_1(\la) \, \cdots \, m_{1+n} (\la)]$ in terms of its entries. From Theorem \ref{minbasis_th}, we know that $M(\la_0) \ne 0$ for any $\lambda_0 \in \overline{\FF}$. Given $\epsilon >0$, define a number $\la_\epsilon$ whose modulus is large enough to satisfy $|m_i (\la_\epsilon) /\la_\epsilon^d| < \epsilon/\sqrt{1+n}$ for $i=1, \ldots, 1+n$. Note that such number exists because all of the scalar polynomials $m_i(\la)$ have degree strictly smaller than $d$. In this case, we construct $\widetilde{M} (\la)$ as follows
$$
\widetilde{M} (\la) = M(\la) - \la^d \, \left[ \frac{m_1 (\la_\epsilon)}{\la_\epsilon^d} \; \cdots \; \frac{m_{1+n} (\la_\epsilon)}{\la_\epsilon^d} \right],
$$
which satisfies $\widetilde{M} (\la_\epsilon) = 0$, so $\widetilde{M} (\la)$ is not a minimal basis, and $\|S_1 (M)- S_1 (\widetilde{M})\|_2 < \epsilon$. This completes the proof of part (a).

Proof of (b). Corollary \ref{cor2:Mdfull} guarantees the existence of a Sylvester matrix $S_k (M)$ with full row rank or, equivalently, with minimal singular value
$\sigma_{(k+d)m} (S_{k}(M)) > 0$. Condition \eqref{eq.smoothneigh1} and Lemma \ref{C} imply
\begin{align*}
\|S_k (M)- S_k (\widetilde{M})\|_2 & = \|S_k (M -\widetilde{M})\|_2 \\ & \leq \sqrt{k} \,
\|S_1 (M -\widetilde{M})\|_2  = \sqrt{k} \,
\|S_1 (M)- S_1 (\widetilde{M})\|_2 \\
& < \sigma_{(k+d)m} (S_{k}(M)),
\end{align*}
which in turns implies that $S_k (\widetilde{M})$ has full row rank by Weyl's perturbation theorem for singular values \cite{stewart-sun}. Then, part (b) follows from applying Corollary \ref{cor2:Mdfull} to $\widetilde{M} (\la)$.
\end{proof}

\begin{remark} \label{rem.onsmoothminbases} The natural choice of $k$ in Theorem \ref{thm.smoothminbases}-(b) is the smallest index $k_0$ for which $S_k(M)$ has full row rank, since in this way the denominator of the right hand side of \eqref{eq.smoothneigh1} is the smallest possible one, which favors a larger estimation of the neighbourhood of robustness. However, note that the numerator plays a nontrivial role and another larger $k$ might be a better choice.
\end{remark}

As a consequence of the proof of Theorem \ref{thm.smoothminbases}-(a) when $M_d=0$, it is obvious that minimal bases are never robust under perturbations that increase their degrees. This simple fact is stated for completeness in the next corollary.
\begin{corollary} \label{cor.largerdegperturb}
Let $M(\la) \in \FF[\la]^{p\times q}$ be any minimal basis with $p < q$ and let $d$ be any integer such that $d > \mbox{\rm deg}(M)$. Then, there exist polynomial matrices in $\FF[\la]^{p \times q}_d$ which are not minimal bases and are as close as we want to $M(\la)$.
\end{corollary}

We have proved in Section \ref{sec.genericity} that generically the polynomial matrices in $\FF[\la]^{m \times (m+n)}_d$ have full-Sylvester-rank, which implies that generically the polynomials in $\FF[\la]^{m \times (m+n)}_d$ are minimal bases with all their row degrees equal to $d$ {\em and with the row degrees of the minimal bases dual to them fully determined by} \eqref{eq.2statfullSylrank}. It is not surprising, due to their genericity, that full-Sylvester-rank polynomial matrices are robust. This is established in Theorem \ref{thm.smoothfullSylvrank} together with estimates of the sizes of the robustness neighbourhoods. We advance that Theorem \ref{thm.smoothfullSylvrank} plays a key role in the perturbation theory of dual minimal bases developed in Section \ref{sec.dualminbases}, which is based on the mentioned idea that the row degrees of the minimal bases dual to full-Sylvester-rank matrices (or, equivalently, the right minimal indices of full-Sylvester-rank matrices) are completely fixed (recall Theorems \ref{thm.propsfullSylvrank}-(a) and \ref{thm.2necsuffullSylvRank}) and, thus, remain constant in a robustness neighborhood of any full-Sylvester-rank matrix. This invariance property is lost in any neighborhood of any minimal basis without full-Sylvester-rank, which implies that perturbation results for dual minimal bases are no longer possible in that scenario, unless perturbations with particular properties are considered.

\begin{theorem} \label{thm.smoothfullSylvrank} Let $M(\la) \in \FF[\la]^{m \times (m+n)}_d$ be a polynomial matrix with full-Sylvester-rank and let $k'$ and $t$ be defined as in \eqref{eq.ktprime}. Then the following statements hold:
\begin{itemize}
\item[\rm (a)] If $k' > 1$ and $t >0$, then every $\widetilde{M} (\la) \in \FF[\la]^{m \times (m+n)}_d$ such that
\[
\|S_1 (M)- S_1 (\widetilde{M})\|_2 < \min \left\{
\frac{\sigma_{(k'-1)(m+n)} (S_{k' -1}(M))}{\sqrt{k'-1}} \, , \,
\frac{\sigma_{(k'+d)m} (S_{k'}(M))}{\sqrt{k'}} \right\} \,
\]
has full-Sylvester-rank.

\item[\rm (b)] If $k' = 1$ or $t = 0$,
then every $\widetilde{M} (\la) \in \FF[\la]^{m \times (m+n)}_d$ such that
\[
\|S_1 (M)- S_1 (\widetilde{M})\|_2 <
\frac{\sigma_{(k'+d)m} (S_{k'}(M))}{\sqrt{k'}} \,
\]
has full-Sylvester-rank.
\end{itemize}
\end{theorem}

\begin{proof} Proof of (a). Recall that, from Lemma \ref{lemm.two-ranks}-(b), $S_{k' -1}(M)$ and $S_{k'}(M)$ have full column and full row rank, respectively, which implies that the smallest singular values of these matrices satisfy $\sigma_{(k'-1)(m+n)} (S_{k' -1}(M))>0$ and $\sigma_{(k'+d)m} (S_{k'}(M))>0$. The assumption in part (a) and Lemma \ref{C} imply
\begin{align*}
\|S_{k'-1} (M)- S_{k' -1} (\widetilde{M})\|_2 & = \|S_{k' -1} (M -\widetilde{M})\|_2 \leq \sqrt{k'-1} \,
\|S_1 (M -\widetilde{M})\|_2 \\ & = \sqrt{k'-1} \,
\|S_1 (M)- S_1 (\widetilde{M})\|_2  < \sigma_{(k' -1)(m+n)} (S_{k'-1}(M))
\end{align*}
and
\[
\|S_{k'} (M)- S_{k'} (\widetilde{M})\|_2 \leq \sqrt{k'} \, \|S_1 (M)- S_1 (\widetilde{M})\|_2 < \sigma_{(k'+d)m} (S_{k'}(M)).
\]
Then, Weyl's perturbation theorem for singular values \cite{stewart-sun} implies that
$S_{k' -1}(\widetilde{M})$ and $S_{k'}(\widetilde{M})$ have full column and full row rank, respectively, and Lemma \ref{lemm.two-ranks}-(b) applied to $\widetilde{M} (\la)$ implies that $\widetilde{M} (\la)$ has full Sylvester rank.

The proof of part (b) is based on Lemma \ref{lemm.two-ranks}-(c) and is analogous to that of part (a). Thus, it is omitted.
\end{proof}

\begin{example} \label{ex1-cont} Let us illustrate Theorem \ref{thm.smoothfullSylvrank} for the pencil in Example \ref{ex1}. In this example,
$m=6$, $n=2$, $d=1$, $k'=3$, and $t=0$. Moreover, $S_3 (M)$ has size $24 \times 24$ and we saw in Example \ref{ex1} that $r_3 = \rank S_3 (M) = 24$. Thus, Lemma \ref{lemm.two-ranks}-(c) guarantees that $M(\la)$ has full-Sylvester-rank and we can apply Theorem \ref{thm.smoothfullSylvrank}-(b) to determine that the neighbourhood of full-Sylvester-rank-robustness is in this example
$$\| S_1 (M)- S_1 (\widetilde{M}) \|_2 <  \frac{\sigma_{(k'+d)m} (S_{k'}(M))}{\sqrt{k'}} = \frac{\sigma_{24}(S_{3} (M))}{\sqrt{3}} \approx 0.2569.$$
Since $\|S_1 (M) \|_2 = \sqrt{2}$, this is a rather large robustness neighborhood.
\end{example}

We finish this section with an improvement of a particular case of Theorem \ref{thm.smoothfullSylvrank}, which is motivated by the following discussion. The estimations in Theorem \ref{thm.smoothfullSylvrank} of the sizes of the neighbourhoods of full-Sylvester-rank-robustness are based on the smallest singular values of $S_{k'-1} (M)$ and $S_{k'} (M)$ (or only $S_{k'} (M)$ when $k'=1$ or $t=0$), which leads to easily computable estimations. However, what really determines the robustness of the full-Sylvester-rank of $M(\la)$ is the smallest of the norms of the perturbations $\Delta M(\la)$ that make  $S_{k'-1} (M + \Delta M)$ or $S_{k'} (M + \Delta M)$ rank deficient, as a consequence of Lemma \ref{lemm.two-ranks}. The point is that this is not measured by the smallest singular values of $S_{k'-1} (M)$ and $S_{k'} (M)$, because these singular values measure the distance of $S_{k'-1} (M)$ and $S_{k'} (M)$ to rank deficient matrices under unstructured perturbations that do not preserve the zero structure of the Sylvester matrices $S_{k'-1} (M + \Delta M)$ and $S_{k'} (M + \Delta M)$. This implies that the estimates of the sizes of the robustness neighborhoods obtained in Theorem \ref{thm.smoothfullSylvrank} are, in general, strictly smaller than the actual sizes of those neighborhoods, except in the case $k' = 1$, i.e., $md \leq n$, corresponding to very ``flat'' polynomial matrices, where the size of the robustness neighborhood is sharp because $S_1(M)$ has no special zero structure. This is stated in the next corollary.

\begin{corollary} \label{cor.sharprobust} Let $M(\la) \in \FF[\la]^{m \times (m+n)}_d$ be a polynomial matrix such that $m d \leq n$. Then the following statements hold:
\begin{itemize}
\item[\rm (a)] $M(\la)$ has full-Sylvester-rank if and only if $S_{1} (M(\la))$ has full row rank.

\item[\rm (b)] Every $\widetilde{M} (\la) \in \FF[\la]^{m \times (m+n)}_d$ such that
$
\|S_1 (M)- S_1 (\widetilde{M})\|_2 < \sigma_{(d+1)m} (S_{1}(M)) \,
$
has full-Sylvester-rank.

\item[\rm (c)] There exists a polynomial matrix $\widetilde{M} (\la) \in \FF[\la]^{m \times (m+n)}_d$ that has not full-Sylvester-rank and satisfies
$
\|S_1 (M)- S_1 (\widetilde{M})\|_2 = \sigma_{(d+1)m} (S_{1}(M)) \, .
$
\end{itemize}
\end{corollary}

\begin{proof} Parts (a) and (b) follows immediately from the fact that $k' =1$ in \eqref{eq.ktprime}, from Lemma \ref{lemm.two-ranks} and Theorem \ref{thm.smoothfullSylvrank}-(b). The matrix polynomial in part (c) is extracted from any rank deficient matrix $A \in \FF^{(d+1)m \times (m+n)}$ such that $\|S_1 (M)- A\|_2 = \sigma_{(d+1)m} (S_{1}(M))$ with the identification $A=S_1 (\widetilde{M})$. One such $A$ is obtained from the singular value decomposition of $S_{1}(M)$ by replacing the smallest singular value by zero.
\end{proof}

\section{Perturbations of minimal bases dual to full-Sylvester-rank matrices} \label{sec.dualminbases} Some applications require to know how the minimal bases dual to a given minimal basis $M(\la)$ change when $M(\la)$ is perturbed. See Section \ref{sec.applications} for more information about such applications. In fact, this problem has been solved in \cite[Section 6.2]{blockKron} for certain particular minimal bases that are the off-diagonal blocks of block Kronecker linearizations of polynomial matrices \cite{blockKron} and has been applied to the backward error analysis of polynomial eigenvalue problems solved via block Kronecker linearizations. The minimal bases considered in \cite[Section 6.2]{blockKron} are very particular instances of full-Sylvester-rank polynomial matrices that have all their row degrees equal to $1$ and all the row degrees of their dual minimal bases equal. The purpose of this section is to solve the corresponding perturbation problem for any polynomial matrix $M(\la)$ that has full-Sylvester-rank. Therefore, the degrees of its dual minimal bases are not always all equal, although they are completely fixed according to \eqref{eq.2statfullSylrank} by the degree and size of $M(\la)$, and so remain constant in a robustness neighborhood of the full-Sylvester-rank property. We emphasize that only for the case of full-Sylvester-rank polynomial matrices such perturbation theory is possible, since in other cases tiny perturbations of $M(\la)$ may change dramatically the row degrees of the minimal bases of its dual space, even in the case that the perturbation of $M(\la)$ is still a minimal basis.

The main result in this section is Theorem \ref{thm.perturbdualbasis}, which deserves some comments before being stated. A key observation on our problem is that there are infinitely many minimal bases $N(\la)$ dual to a given full-Sylvester-rank polynomial matrix $M(\la)$ and a key feature of the solution we propose is that Theorem \ref{thm.perturbdualbasis} provides the same perturbation bound for a {\em relative} change of all of them. However, the allowed norms of the perturbations of $M(\la)$ do depend on the considered dual basis $N(\la)$ in order to guarantee that the perturbed $\widetilde{N} (\la)$ found in the proof of Theorem \ref{thm.perturbdualbasis} is indeed a minimal basis. In Theorem \ref{thm.perturbdualbasis}, the Frobenius norm \cite{stewart-sun} $\|S_k (P)\|_F$ of the Sylvester matrices of a polynomial matrix $P(\la)$ is used, in addition to the spectral or 2-norm used in Section \ref{sec.smoothness}. The reason of using the Frobenius norm is that the equality $\|S_k (P)\|_F =\|S_k (P^T)\|_F$ holds, while it is not valid in general for the spectral norm, and this equality makes the proof simpler and the bounds cleaner.

\begin{theorem} \label{thm.perturbdualbasis} Let $M(\la) \in \FF[\la]^{m \times (m+n)}_d$ be a polynomial matrix with full-Sylvester-rank, let $k'$ and $t$ be defined as in \eqref{eq.ktprime}, and let $N(\la) \in \FF[\la]^{n \times (m+n)}_{k'}$ be a minimal basis dual to $M(\la)$ with highest-row-degree coefficient matrix $N_{hr}\in \FF^{n \times (m+n)}$. Moreover, let us define the quantities $\theta_1 (M)$ and $\theta_2 (M)$ as follows:
\begin{itemize}
\item[\rm (a)] If $k' >1$ and $t>0$
\begin{align*}
\theta_1 (M) & := \min \left\{
\frac{\sigma_{(k'-1)(m+n)} (S_{k' -1}(M))}{\sqrt{k'-1}} \, , \,
\frac{\sigma_{(k'+d)m} (S_{k'}(M))}{\sqrt{k'}} \, , \,
\frac{\sigma_{(k'+1+d)m} (S_{k'+1}(M))}{\sqrt{k'+1}}
\right\} \, , \\
\theta_2 (M) & := \min \left\{
\frac{\sigma_{(k'+d)m} (S_{k'}(M))}{\sqrt{k'}} \, , \,
\frac{\sigma_{(k'+1+d)m} (S_{k'+1}(M))}{\sqrt{k'+1}}
\right\} \, ;
\end{align*}

\item[\rm (b)] If $k' = 1$ and $t >0$,
\begin{align*}
\theta_1 (M) & = \theta_2 (M) := \min \left\{
\sigma_{(d+1)m} (S_{1}(M)) \, , \,
\frac{\sigma_{(d+2)m} (S_{2}(M))}{\sqrt{2}}
\right\} \, ;
\end{align*}

\item[\rm (c)] If $t =0$
\begin{align*}
\theta_1 (M) & := \min \left\{
\frac{\sigma_{(k'+d)m} (S_{k'}(M))}{\sqrt{k'}} \, , \,
\frac{\sigma_{(k'+1+d)m} (S_{k'+1}(M))}{\sqrt{k'+1}}
\right\} , \\
\theta_2 (M) & :=
\frac{\sigma_{(k'+1+d)m} (S_{k'+1}(M))}{\sqrt{k'+1}}
\, .
\end{align*}
\end{itemize}
Then, every $\widetilde{M} (\la) \in \FF[\la]^{m \times (m+n)}_d$ such that
\begin{equation} \label{eq.dualpert1}
\|S_1 (M)- S_1 (\widetilde{M})\|_2 < \frac{1}{2} \cdot \theta_1 (M) \cdot  \frac{\sigma_n (N_{hr})}{\|S_1 (N)\|_F} \,
\end{equation}
has full-Sylvester-rank and has a dual minimal basis $\widetilde{N} (\la) \in \FF[\la]^{n \times (m+n)}_{k'}$ that satisfies
\begin{equation} \label{eq.dualpert2}
\frac{\|S_1 (N) - S_1 (\widetilde{N})\|_F}{\|S_1 (N)\|_F} \leq \frac{2}{\theta_2 (M)} \cdot \|S_1 (M)- S_1 (\widetilde{M})\|_2 \, .
\end{equation}
In addition, if $t=0$, then all the row degrees of $\widetilde{N}(\la)$ and $N(\la)$ are equal to $k'$.
\end{theorem}

\begin{proof} We only prove the most difficult case, which is (a), i.e., when $k' > 1$ and $t >0$. The proof of the case (b) is essentially the same as that of (a),
with the differences coming from the differences between parts (a) and (b) of Theorem \ref{thm.smoothfullSylvrank}. The proof of (c) is considerably simpler, as it is briefly explained at the end of this proof. In order to perform the proof of case (a), note first that $N_{hr}$ is a submatrix of $S_1 (N)$. Thus, we get
\begin{equation} \label{eq.sigmahr}
\frac{\sigma_n (N_{hr})}{\|S_1 (N)\|_F}  \leq
\frac{\sigma_n (N_{hr})}{\|S_1 (N)\|_2} \leq \frac{\sigma_n (N_{hr})}{\|N_{hr}\|_2} \leq 1 \, .
\end{equation}
This inequality and \eqref{eq.dualpert1} imply that $\|S_1 (M)- S_1 (\widetilde{M})\|_2$ is bounded as in Theorem \ref{thm.smoothfullSylvrank}-(a). Therefore, any $\widetilde{M} (\la)$ that satisfies \eqref{eq.dualpert1} has full-Sylvester-rank and every of its dual minimal bases $\widetilde{N}(\la)$ has $t$ row degrees equal to $k'-1$, and $n-t$ row degrees equal to $k'$, as it also happens for the minimal basis $N(\la)$ dual to $M(\la)$ (recall Theorem \ref{thm.2necsuffullSylvRank}). As a consequence we can write, modulo row permutations,
\begin{equation} \label{eq.dualpartdegree}
N(\la) = \left[ \begin{array}{c} X(\la) \\ Y(\la) \end{array} \right]
\quad \mbox{and} \quad
\widetilde{N}(\la) = \left[ \begin{array}{c} X(\la) \\ Y(\la) \end{array} \right]  + \left[ \begin{array}{c} \Delta X(\la) \\ \Delta Y(\la) \end{array} \right],
\end{equation}
where $X(\la), X(\la)+ \Delta X(\la) \in \FF[\la]_{k'-1}^{t \times (m+n)}$ have both all their row degrees equal to $k'-1$, and $Y(\la), Y(\la)+ \Delta Y(\la) \in \FF[\la]_{k'}^{(n-t) \times (m+n)}$ have both all their row degrees equal to $k'$. In the rest of the proof, among the infinitely many minimal bases $\widetilde{N} (\la)$ dual to $\widetilde{M} (\la)$, we will determine one that satisfies \eqref{eq.dualpert2} by determining the corresponding $\Delta X(\la)$ and $\Delta Y(\la)$. To this purpose, note that the equation $\widetilde{M} (\la) \, \widetilde{N} (\la)^T = 0$ defining dual minimal bases is equivalent with the notation in \eqref{eq.dualpartdegree} to the pair of equations
\begin{equation} \label{eq.2dualpartdegree}
\widetilde{M} (\la) \, (X(\la)^T+ \Delta X(\la)^T) = 0 \quad \mbox{and} \quad
\widetilde{M} (\la) \, (Y(\la)^T+ \Delta Y(\la)^T) = 0 \, .
\end{equation}
By introducing the notation $\widetilde{M} (\la) = M (\la) + \Delta M (\la)$ and by using $$M(\la) N(\la)^T = [ M(\la) X(\la)^T  \; M(\la) Y(\la)^T ] = 0,$$ the equations in \eqref{eq.2dualpartdegree} can be written in the following equivalent form
\[
\widetilde{M} (\la) \, \Delta X(\la)^T = - \Delta M (\la) \,  X(\la)^T\quad \mbox{and} \quad
\widetilde{M} (\la) \, \Delta Y(\la)^T = - \Delta M (\la) \, Y(\la)^T \, ,
\]
which in terms of the corresponding Sylvester matrices are equivalent to the equations
\begin{align} \label{eq.3dualpartdegree}
S_{k'} (\widetilde{M}) \, S_1 (\Delta X^T) & = - S_{k'}(\Delta M) \,  S_1(X^T)\quad \mbox{and} \\ \label{eq.4dualpartdegree}
S_{k' +1} (\widetilde{M}) \, S_1 (\Delta Y^T) & = - S_{k'+1} (\Delta M) \, S_1(Y^T) \,
\end{align}
for the unknowns $S_1 (\Delta X^T)$ and $S_1 (\Delta Y^T)$. Since $\widetilde{M}(\la)$ has full-Sylvester-rank, the Sylvester matrices $S_{k'} (\widetilde{M})$ and $S_{k' +1} (\widetilde{M})$ have both full row rank and the equations \eqref{eq.3dualpartdegree} and \eqref{eq.4dualpartdegree} are consistent. In addition, for each of them, its minimum Frobenius norm solution is given by
\begin{align} \label{eq.5dualpartdegree}
 S_1 (\Delta X^T) & = - S_{k'} (\widetilde{M})^\dagger \,S_{k'}(\Delta M) \,  S_1(X^T)\quad \mbox{and} \\ \label{eq.6dualpartdegree}
 S_1 (\Delta Y^T) & = - S_{k' +1} (\widetilde{M})^\dagger \, S_{k'+1} (\Delta M) \, S_1(Y^T) \, ,
\end{align}
where $(\cdot)^\dagger$ stands for the Moore-Penrose pseudoinverse of a matrix. From \eqref{eq.5dualpartdegree}, Lemma \ref{C}, Weyl's perturbation theorem for singular values, and \eqref{eq.dualpert1}, we get
\begin{align} \nonumber
\|S_1 (\Delta X^T) \|_F & \leq \frac{\|S_{k'}(\Delta M) \|_2}{\sigma_{(k'+d)m} (S_{k'} (\widetilde{M}))} \, \|S_1(X^T)\|_F \leq
\frac{\sqrt{k'} \, \|S_{1}(\Delta M) \|_2}{\sigma_{(k'+d)m} (S_{k'} (M))- \|S_{k'} (\Delta M) \|_2} \, \|S_1(X^T)\|_F \\ \nonumber
& \leq
\frac{\sqrt{k'} \, \|S_{1}(\Delta M) \|_2}{\sigma_{(k'+d)m} (S_{k'} (M))- \sqrt{k'} \, \|S_{1} (\Delta M) \|_2} \, \|S_1(X^T)\|_F \leq
\frac{\sqrt{k'} \, \|S_{1}(\Delta M) \|_2}{\frac{1}{2} \, \sigma_{(k'+d)m} (S_{k'} (M))} \, \|S_1(X^T)\|_F \\ \label{eq.7dualpartdegree}
& = \frac{2 \, \|S_{1}(\Delta M) \|_2}{\frac{1}{\sqrt{k'}} \, \sigma_{(k'+d)m} (S_{k'} (M))} \, \|S_1(X^T)\|_F \, .
\end{align}
Analogously, we get from
\eqref{eq.6dualpartdegree}
\begin{equation} \label{eq.8dualpartdegree}
\|S_1 (\Delta Y^T) \|_F \, \leq \, \frac{2 \, \|S_{1}(\Delta M) \|_2}{\frac{1}{\sqrt{k'+1}} \, \sigma_{(k'+1+d)m} (S_{k'+1} (M))} \, \|S_1(Y^T)\|_F \, .
\end{equation}
Observe that if $\Delta X (\la) = \sum_{i=0}^{k'-1} \Delta X_i \, \la^i$ and
$\Delta Y (\la) = \sum_{i=0}^{k'} \Delta Y_i \, \la^i$, then
\begin{align*}
\|S_1(N^T) - S_1 (\widetilde{N}^T)\|_F & = \|S_1([\Delta X(\la)^T \; \Delta Y(\la)^T ] )\|_F = \|
\left[
\begin{array}{cc}
\Delta X_0^T & \Delta Y_0^T \\
\vdots & \vdots \\
\Delta X_{k'-1}^T & \Delta Y_{k'-1}^T \\
0  & \Delta Y_{k'}^T \\
\end{array}
\right]
\|_F \\
& = \sqrt{\|S_1 (\Delta X^T) \|_F^2 + \|S_1 (\Delta Y^T) \|_F^2} \, ,
\end{align*}
which combined with \eqref{eq.7dualpartdegree} and \eqref{eq.8dualpartdegree} yields
\[
\|S_1(N^T) - S_1 (\widetilde{N}^T)\|_F  \leq \frac{2 \, \|S_{1}(\Delta M) \|_2}{\theta_2 (M)} \, \sqrt{\|S_1 (X^T) \|_F^2 + \|S_1 (Y^T) \|_F^2}
=
\frac{2 \, \|S_{1}(\Delta M) \|_2}{\theta_2 (M)} \, \|S_1 (N^T) \|_F \, .
\]
Since $\|S_k (P)\|_F =\|S_k (P^T)\|_F$ for any polynomial matrix, this inequality is \eqref{eq.dualpert2}. It remains to prove that $\widetilde{N} (\la)$ is a minimal basis, since we have only proved so far that $\widetilde{M} (\la) \, \widetilde{N} (\la)^T =0$. In order to prove that $\widetilde{N} (\la)$ is a minimal basis, we will show that
\begin{equation} \label{eq.ZNhr}
Z := N_{hr} + \left[ \begin{array}{c} \Delta X_{k'-1} \\ \Delta Y_{k'} \end{array} \right]
\end{equation}
has full row rank, because in this case \eqref{eq.dualpartdegree} guarantees that $Z$ is the highest-row-degree coefficient matrix of $\widetilde{N}(\la)$, corresponding to $t$ row degrees equal to $k'-1$ and $n-t$ row degrees equal to $k'$, and that $\widetilde{N}(\la)$ is row reduced. Thus $\widetilde{N}(\la)$ has full row normal rank (recall Remark \ref{rem:unicity}) and since its row degrees are those of a minimal basis dual to $\widetilde{M} (\la)$ and it satisfies $\widetilde{M} (\la) \, \widetilde{N} (\la)^T =0$, we get that $\widetilde{N}(\la)$  must indeed be a minimal basis dual to $\widetilde{M} (\la)$. The fact that $Z$ in \eqref{eq.ZNhr} has full row rank follows from \eqref{eq.dualpert2} (which has been already proved) combined with the hypothesis \eqref{eq.dualpert1} which imply
\[
\| \left[ \begin{array}{c} \Delta X_{k'-1} \\ \Delta Y_{k'} \end{array} \right] \|_2 \leq \| \left[ \begin{array}{c} \Delta X_{k'-1} \\ \Delta Y_{k'} \end{array} \right] \|_F \leq \|S_1 (\widetilde{N}) - S_1 (N) \|_F < \frac{\theta_1 (M)}{\theta_2 (M)} \, \sigma_n (N_{hr})
\leq \sigma_n (N_{hr}).
\]
Taking into account that $N_{hr}$ has full row rank, this inequality implies that $Z$ has also full row rank. The proof of the case (a) is completed.

Observe that in the case $t=0$ considered in (c), the matrices $X(\la)$ and $\Delta X(\la)$ are not present in \eqref{eq.dualpartdegree}, i.e., the partition into two block rows corresponding to polynomial matrices of different degrees $k'-1$ and $k'$ is no longer needed, since all the row degrees of $N(\la)$ and $\widetilde{N} (\la)$ are equal to $k'$. The proof is considerably simpler and follows from the one above for the case (a) by ignoring all equations involving $X(\la)$ and $\Delta X (\la)$.
\end{proof}

It is interesting to emphasize with respect to Theorem \ref{thm.perturbdualbasis} that if $t>0$, then the minimal bases $N(\la)$ dual to full-Sylvester-rank matrices $M(\la) \in \FF[\la]_{d}^{m\times (m+n)}$ are never robust under arbitrary perturbations in $\FF[\la]_{k'}^{n\times (m+n)}$ since the row degrees of $N(\la)$ are not all equal (recall Theorem \ref{thm.smoothminbases}). However, Theorem \ref{thm.perturbdualbasis} proves that such $N(\la)$ can behave smoothly if the perturbations are restricted to be in the set of minimal bases dual to the full-Sylvester-rank matrices satisfying \eqref{eq.dualpert1}. In addition, if $t=0$ then we prove in Theorem \ref{thm.dualfullSylvrank} that any such $N(\la)$ has also full-Sylvester-rank and, therefore, $N(\la)$ admits a robustness result analogous to Theorem \ref{thm.smoothfullSylvrank}-(b) and a perturbation theorem similar to Theorem \ref{thm.perturbdualbasis}-(c) with respect to its dual minimal basis $M(\la)$.

\begin{theorem} \label{thm.dualfullSylvrank} Let $M(\la) \in \FF[\la]^{m \times (m+n)}_d$ be a polynomial matrix with full-Sylvester-rank, let $k'$ and $t$ be defined as in \eqref{eq.ktprime}, and let $N(\la) \in \FF[\la]^{n \times (m+n)}_{k'}$ be a minimal basis dual to $M(\la)$. Then, $N(\la)$ has full-Sylvester-rank if and only if $t=0$, i.e., if and only if all the row degrees on $N(\la)$ are equal.
\end{theorem}

\begin{proof}
If $t >0$, then the row degrees of $N(\la)$ take the two values $k'-1$ and $k'$, as a consequence of Theorem \ref{thm.2necsuffullSylvRank}, and $N(\la)$ has not full-Sylvester-rank according to Theorem \ref{thm.propsfullSylvrank}-(b) (applied to $N(\la)$).

If $t=0$, then the quantities in \eqref{eq.ktprime} corresponding to $N(\la)$ are $k'_N = \lceil nk'/m\rceil = d$ and $t'_N = 0$. In addition, $N(\la)$ has exactly $m$ right minimal indices equal to $d$ because these are the row degrees of $M(\la)$. Therefore, Theorem \ref{thm.2necsuffullSylvRank} implies that $N(\la)$ has full-Sylvester-rank.
\end{proof}

We finish this section with another nice feature of the case $t=0$. In this case, also the reversal polynomials of $M(\la)$ and $N(\la)$ have full-Sylvester-rank and, therefore, admit robustness and perturbation results as those in Theorem \ref{thm.smoothfullSylvrank}-(b) and Theorem \ref{thm.perturbdualbasis}-(c).

\begin{theorem} \label{thm.revfullSylvrank} Let $M(\la) \in \FF[\la]^{m \times (m+n)}_d$ be a polynomial matrix with full-Sylvester-rank, let $k'$ and $t$ be defined as in \eqref{eq.ktprime}, let $N(\la) \in \FF[\la]^{n \times (m+n)}_{k'}$ be a minimal basis dual to $M(\la)$, and let
$$
\mbox{\rm rev}_d M(\la) := \la^d \, M\left(\frac{1}{\la} \right) \quad \mbox{and} \quad
\mbox{\rm rev}_{k'} N(\la) := \la^{k'} \, N\left(\frac{1}{\la} \right)
.$$
If $t=0$, then $\mbox{\rm rev}_d M(\la)$ and $\mbox{\rm rev}_{k'} N(\la)$ are dual minimal bases and both have full-Sylvester-rank.
\end{theorem}

\begin{proof}
Since $t=0$, all the row degrees of $N(\la)$ are equal to $k'$ (recall Theorem \ref{thm.2necsuffullSylvRank}). Then, \cite[Theorem 2.7]{blockKron} implies that $\mbox{\rm rev}_d M(\la)$ and $\mbox{\rm rev}_{k'} N(\la)$ are dual minimal bases with all their row degrees equal, respectively, to $d$ and $k'$. The fact that both have full-Sylvester-rank then follows from Theorem \ref{thm.2necsuffullSylvRank}, since the right minimal indices of $\mbox{\rm rev}_d M(\la)$ (resp. $\mbox{\rm rev}_{k'} N(\la)$) are the row degrees of $\mbox{\rm rev}_{k'} N(\la)$ (resp. $\mbox{\rm rev}_d M(\la)$).
\end{proof}

\section{The classical rank conditions for a minimal basis revisited} \label{sec.classrevisited} In this section we return to the classical characterization of a minimal basis given in Theorem \ref{minbasis_th}, originally presented in \cite{For75}, in the case of minimal bases over $\mathbb{C}$ that are robust under arbitrary perturbations with the same or smaller degree, which according to Theorem \ref{thm.smoothminbases} are those whose row degrees are all equal. These minimal bases include all polynomial matrices with full-Sylvester-rank and also other polynomial matrices, as it is explained in Remark \ref{rem.otherminbases}. More precisely, our goal in this section is to prove that in the case of robust minimal bases all of the infinitely many constant matrices whose ranks are involved in Theorem \ref{minbasis_th} have minimum singular values bounded below by a common number determined by one of the Sylvester matrices of the considered minimal basis, and, so, none of such matrices is extremely closed to be rank deficient. This is proved in the following theorem.

\begin{theorem} \label{thm.classrevisit} Let $M(\la) = M_0 + M_1 \la + \cdots + M_d \la^d \in \CC[\la]^{m \times (m+n)}$ be a minimal basis with $\rank M_d =m$, let $S_k$ be its Sylvester matrices for $k=1, 2, \ldots$, and let $d'$ be the largest right minimal index of $M(\la)$ or, equivalently, let $d'$ be the degree of any minimal basis dual to $M(\la)$. Then
\[
\sigma_{(d+d')m}(S_{d'}) \leq  \inf_{\la_0\in \CC} \sigma_m (M(\la_0)) \quad \mbox{\rm and} \quad \sigma_{(d+d')m}(S_{d'}) \leq \sigma_m (M_d) \,.
\]

\end{theorem}

\begin{proof}
From Corollary \ref{cor2:Mdfull}, we obtain that $S_{d'}$ has full row rank. Therefore, its smallest singular value is larger than zero, i.e., $\sigma_{(d+d')m}(S_{d'}) >0$. We use in this proof the well known fact \cite{HoJo, stewart-sun} that any matrix $A\in \CC^{p \times q}$ with $p \leq q$ satisfies
\begin{equation} \label{eq.minsingvalgen}
\sigma_p (A) = \min_{0 \ne x \in \CC^{p}} \frac{\|A^* x \|_2}{\|x\|_2} =
\min_{0 \ne x \in \CC^{p}} \frac{\|x^* A \|_2}{\|x\|_2} \, .
\end{equation}
This result applied to $A = S_{d'}$ immediately implies that $\sigma_{(d+d')m}(S_{d'}) \leq \sigma_m (M_d)$, since the last block row of $S_{d'}$ is $[0 \; M_d]$ and one can choose in \eqref{eq.minsingvalgen} vectors $x$ with entries not corresponding to this last block row equal to zero. To prove the first inequality in Theorem \ref{thm.classrevisit} note that the polynomial matrix $M(\la)$ satisfies the following equality between polynomial matrices:
\begin{equation} \label{conv} \left[ \begin{array}{ccccc} I & I \la & I \la^2 & \ldots & I \la^{d+d'-1} \end{array}\right]
   \underbrace{\left[ \begin{array}{ccccc} M_0 \\ M_1 & M_0 \\
   \vdots & M_1 &  \ddots \\
   M_d & & \ddots & M_0 \\
   0 & M_d & & M_1\\
   \vdots & \ddots & \ddots & \vdots \\
   0 & \ldots & 0 & M_d
   \end{array}\right]}_{d' \; \mathrm{blocks}} =
   M(\la) \left[ \begin{array}{ccccc} I & I\la & I\la^2 & \ldots & I\la^{d'-1} \end{array}\right],
\end{equation}
where the identity matrices in the left hand side are equal to $I_m$ and in the right hand side are equal to $I_{m+n}$. Theorem \ref{minbasis_th} implies that $\sigma_{0} := \sigma_m (M(\la_0)) >0$ for any $\la_0 \in \CC$. Let $u_0\in \CC^m$ and $v_0\in \CC^{(m+n)}$ be left and right singular vectors of $M(\la_0)$ corresponding to $\sigma_0$, that is $\|u_0\|_2=\|v_0\|_2=1$ and  $u_0^* \, M(\la_0)=\sigma_0 \, v_0^*$. Then it follows from \eqref{conv} that
$$ (\left[ \begin{array}{ccccc} 1 & \la_0 & \la^2_0 & \ldots & \la^{d+d'-1}_0 \end{array}\right] \otimes u_0^*)
   \underbrace{\left[ \begin{array}{ccccc} M_0 \\ M_1 & M_0 \\
   \vdots & M_1 &  \ddots \\
   M_d & & \ddots & M_0 \\
   0 & M_d & & M_1\\
   \vdots & \ddots & \ddots & \vdots \\
   0 & \ldots & 0 & M_d
   \end{array}\right]}_{d' \; \mathrm{blocks}} =
   \sigma_0(\left[ \begin{array}{ccccc} 1 & \la_0 & \la^2_0 & \ldots & \la^{d'-1}_0 \end{array}\right] \otimes v_0^*).
$$
From this last equality and \eqref{eq.minsingvalgen} applied to $A=S_{d'}$, we get
$$
\sigma_{(d+d')m}(S_{d'}) \le \sigma_0
\sqrt{\frac{\sum_{i=1}^{d'} |\la_0|^{2(i-1)} }{\sum_{i=1}^{d+d'}|\la_0|^{2(i-1)}}}\le \sigma_0 = \sigma_m (M(\la_0)).
$$
Since this holds for all $\la_0 \in \CC$, the result is proved.
\end{proof}

Theorem \ref{thm.classrevisit} is consistent with the robustness result proved in Theorem \ref{thm.smoothminbases}-(b) that guarantees that all polynomial matrices in a neighborhood of a minimal basis $M(\la)$ with all its row degrees equal are also minimal bases, as long as we restrict to the space of polynomial matrices with degree less than or equal to $\deg (M)$. The point is that if $\inf_{\la_0\in \CC} \sigma_m (M(\la_0)) =0$ (note that one can think this might happen although $\sigma_m (M(\mu)) >0$ for each $\mu$), then it would be possible to construct an $\widetilde{M} (\la)$ as close as we want to $M(\la)$ with $\sigma_m (\widetilde{M} (\mu)) =0$ for some particular $\mu$ and such $\widetilde{M} (\la)$ would not be a minimal basis by Theorem \ref{minbasis_th}.

\section{Applications to backward error analyses} \label{sec.applications} In this section, we discuss  very briefly two concrete applications of some of the results in this paper and develop with detail a third application. The goals of this third application are to illustrate, first, how the results in this paper allow us to establish a framework for performing backward error analyses of polynomial eigenvalue problems solved by applying backward stable algorithms to strong linearizations or $\ell$-ifications, and, second, that considerable nontrivial work is still needed to analyze the behaviour of particular linearizations or $\ell$-ifications within this framework.

The first application is to extend the backward error analysis performed in \cite[Section 6]{blockKron} for polynomial eigenvalue problems solved via the block Kronecker linearizations introduced in \cite[Section 5]{blockKron} to the much more general class of strong block minimal bases (SBMB) linearizations of polynomial matrices introduced in \cite[Section 3]{blockKron}, whose construction is based on polynomial matrices with full-Sylvester-rank of degree $1$ whose right minimal indices are all equal. This is a challenging problem with very interesting applications, since these SBMB linearizations include, among many others, strong linearizations of polynomial matrices expressed in non-monomial bases  \cite{lawrence-perez-cheby,robol}, all Fiedler-like linearizations (modulo permutations) published so far in the literature \cite{buenoREU}, as well as some interesting vector spaces of potential strong linearizations of polynomial matrices \cite{fassbendersaltenberger}.

A second application is related to the fact that the SBMB and block Kronecker linearizations of polynomial matrices recently introduced in \cite{blockKron} can be extended to SBMB and block Kronecker {\em $\ell$-ifications} of polynomial matrices \cite{BMBelifications}, which are constructed by using polynomial matrices with full-Sylvester-rank of degree $\ell$ and whose right minimal indices are all equal. The reader is referred to \cite{DDM,DDV-l-ifications} for formal definitions of the concept of an $\ell$-ification of a polynomial matrix. In this paper it is enough to know that an $\ell$-ification of a polynomial matrix $P(\la)$ is another polynomial matrix of degree $\ell$ that has the same (finite and infinite) elementary divisors, the same number of left, and the same number of right minimal indices as $P(\la)$. The results in \cite{BMBelifications} pose naturally the question whether or not the backward error analysis for block Kronecker linearizations in \cite[Section 6]{blockKron} can be extended to the classes of $\ell$-ifications introduced in \cite{BMBelifications}. This seems to be a hard problem whose potential solution will be related for sure to the results in this paper.

The third application in this section is to perform a backward error analysis of polynomial eigenvalue problems solved by applying any backward stable algorithm to the strong $\ell$-fications introduced in \cite{DDV-l-ifications}, which include also many interesting strong linearizations when $\ell =1$. We develop this application with detail and for this purpose it is convenient to revise the most important results from \cite{DDV-l-ifications}. Theorem \ref{thm.stronglifications} was proved in \cite[Theorems 4.1 and 4.5]{DDV-l-ifications}. The statement is essentially the one in \cite{DDV-l-ifications} although the notation has been adapted to fit the one used in previous sections of this paper. Theorem \ref{thm.stronglifications} is valid in an arbitrary field, while the remaining results in this section are valid only when $\FF = \RR$ or $\FF = \CC$.
\begin{theorem} \label{thm.stronglifications} Let
\begin{equation} \label{eq.ilification}
L(\la) = \left[ \begin{array}{c} K(\la) \\ M(\la) \end{array} \right] \in \FF[\la]_\ell^{(p+m) \times (m+n)}
\end{equation}
be a matrix polynomial of degree $\ell$ such that $M(\la) \in \FF[\la]^{m\times (m+n)}$ is a minimal basis with all its row degrees equal to $\ell$ and such that the row degrees of any minimal basis dual to $M(\la)$ are all equal to $k' = m \ell/n$, i.e., all the right minimal indices of $M(\la)$ are equal. If $N(\la) \in \FF[\la]^{n\times (m+n)}_{k'}$ is any minimal basis dual to $M(\la)$, then $L(\la)$ is a strong $\ell$-ification of the polynomial matrix
\begin{equation} \label{eq.invellif}
P(\la) = K(\la) N(\la)^T \in \FF[\la]^{p \times n} \, ,
\end{equation}
considered as a polynomial of degree less than or equal to $\ell + k'$.
In addition, the right minimal indices of $L(\la)$ are obtained from the right minimal indices of $P(\la)$ by adding to each of them $k'$, while the left minimal indices of $L(\la)$ are equal to those of $P(\la)$.
\end{theorem}

Observe that since $k' = m \ell/n$ must be an integer only some combinations of $m$, $n$, and $\ell$ are allowed in Theorem \ref{thm.stronglifications}. Recall also that the value $k'$ of the row degrees of $N(\la)$ is fixed to be precisely $k' = m \ell/n$ by Theorem \ref{thm:dualbasis} and the requirements that all the row degrees of $M(\la)$ are equal to $\ell$ and that all the row degrees of $N(\la)$ are equal. Theorem \ref{thm.stronglifications} establishes a wide framework for constructing strong $\ell$-ifications of matrix polynomials. In particular, this framework includes the Frobenius-like strong $\ell$-ifications introduced for first time in \cite[Section 5.2]{DDM}. Moreover, in \cite[Section 4]{DDV-l-ifications}, it is explained how to determine $K(\la)$ from \eqref{eq.invellif} when $P(\la)$ and $N(\la)$ are given and some interesting examples of strong $\ell$-ifications are presented there.

It is interesting to emphasize that with the new results introduced in this paper {\em the polynomial matrix $M(\la)$ with the properties in Theorem \ref{thm.stronglifications} is just a matrix with full-Sylvester-rank with $t=0$, i.e., with its right minimal indices all equal to $k'$}, which follows immediately from Theorem \ref{thm.2necsuffullSylvRank} and the fact that $k'$ is an integer. This simple observation allows us to prove very easily the perturbation result for strong $\ell$-ifications presented in Theorem \ref{thm.pertstronglifications} by using Theorem \ref{thm.perturbdualbasis}-(c). Recall in this context that $S_1 (Q)$ denotes the first Sylvester matrix of the polynomial matrix $Q(\la)$ and that $\|S_1 (Q)\|_F$ and $\|S_1 (Q)\|_2$ denote the corresponding Frobenius and spectral norms, respectively.

\begin{theorem} \label{thm.pertstronglifications} Let
\[
L(\la) = \left[ \begin{array}{c} K(\la) \\ M(\la) \end{array} \right] \in \FF[\la]_\ell^{(p+m) \times (m+n)}
\]
be a strong $\ell$-ification of the polynomial matrix
$
P(\la) = K(\la) N(\la)^T \in \FF[\la]^{p \times n} \, ,
$
with the properties and notation stated in Theorem \ref{thm.stronglifications}. If
\[
\Delta L(\la) = \left[ \begin{array}{c} \Delta K(\la) \\ \Delta M(\la) \end{array} \right] \in \FF[\la]_\ell^{(p+m) \times (m+n)}
\]
is another polynomial matrix of degree at most $\ell$ partitioned as $L(\la)$ and such that
\begin{equation} \label{eq.1condl-ification}
\|S_1 (\Delta M) \|_2 < \frac{1}{2} \cdot \theta_1 (M) \cdot \frac{\sigma_n (N_{hr})}{\|S_1 (N)\|_F},
\end{equation}
where $\theta_1 (M)$ is defined as in Theorem \ref{thm.perturbdualbasis}-(c) with $d$ replaced by $\ell$ and $N_{hr}$ is the highest-row-degree coefficient matrix of $N(\la)$, then $L(\la)+ \Delta L(\la)$ is a strong $\ell$-ification of a polynomial matrix $P(\la) + \Delta P(\la)$ of degree less than or equal to $\ell + k'$ such that
\begin{equation} \label{eq.1finalbound}
\frac{\|S_1(\Delta P)\|_F}{\|S_1(P)\|_F} \leq \, \min\{\sqrt{k' +1} \, ,\, \sqrt{\ell +1} \} \, \, C_{P,L} \,\, \frac{\|S_1(\Delta L)\|_F}{\|S_1(L)\|_F},
\end{equation}
with
\begin{equation} \label{eq.2finalbound}
C_{P,L} = \frac{\|S_1(L)\|_F}{\|S_1(P)\|_F} \, \|S_1(N)\|_F \,
\left( 1
+ \frac{2 \, \sqrt{k' +1}}{\sigma_{(k' +1 +\ell)m} (S_{k' +1} (M))} \, (\|S_1(K)\|_F + \|S_1(\Delta K)\|_F)\, \right) \,.
\end{equation}
In addition, the right minimal indices of $L(\la) + \Delta L(\la)$ are obtained from the right minimal indices of $P(\la) + \Delta P(\la)$ by adding to each of them $k'$, while the left minimal indices of $L(\la) + \Delta L(\la)$ are equal to those of $P(\la) + \Delta P(\la)$.
\end{theorem}

\begin{proof} The condition \eqref{eq.1condl-ification} is just \eqref{eq.dualpert1} in Theorem \ref{thm.perturbdualbasis}. Therefore, according to Theorem \ref{thm.perturbdualbasis}-(c,) the submatrix $M(\la) + \Delta M(\la) \in \FF[\la]_\ell^{m \times (m+n)}$ of $L(\la) + \Delta L(\la)$ has full-Sylvester-rank, has parameter $t=0$ (where $t$ and $k'$ are defined as in \eqref{eq.ktprime} with $d$ replaced by $\ell$), i.e., all the row degrees of any minimal basis dual to $M(\la) + \Delta M(\la)$ are equal to $k' = m \ell / n$, and there exists one of such dual minimal bases $N(\la) + \Delta N(\la)$ that satisfies \eqref{eq.dualpert2}. Thus, Theorem \ref{thm.stronglifications} applied to $L(\la) + \Delta L(\la)$ implies that $L(\la) + \Delta L(\la)$ is a strong $\ell$-ification of the polynomial matrix
\begin{equation} \label{eq.polypertl-ification}
P(\la) + \Delta P(\la) = (K(\la) + \Delta K(\la)) \, (N(\la) + \Delta N(\la))^T \,
\end{equation}
and that the relationship between the minimal indices of $L(\la) + \Delta L(\la)$ and $P(\la) + \Delta P(\la)$ is the one in the statement. From \eqref{eq.polypertl-ification}, we get
\begin{equation} \label{eq.2polypertl-ification}
\Delta P(\la) = \Delta K(\la) \, N(\la)^T
+ K(\la) \, \Delta N(\la)^T
+ \Delta K(\la) \, \Delta N(\la)^T \, .
\end{equation}
In the rest of the proof, we bound $\|S_1(\Delta P)\|_F$ for $\Delta P(\la)$ in
\eqref{eq.2polypertl-ification} in order to get \eqref{eq.1finalbound} and \eqref{eq.2finalbound}. To this purpose note that for any polynomial matrix $Q(\la)$ the norm $\| S_1(Q) \|_F$ coincides with the norm $\| Q(\la) \|_F$ introduced in \cite[Definition 2.15]{blockKron}. Therefore, we can use the inequalities in \cite[Lemma 2.16]{blockKron} which, together with \eqref{eq.dualpert2} in the case (c) in Theorem \ref{thm.perturbdualbasis} and the inequality $\|S_1 (\Delta M)\|_2 \leq \|S_1 (\Delta M)\|_F$, lead directly to \eqref{eq.1finalbound} with $C_{P,L}$ replaced by
\begin{align*}
\widehat{C}_{P,L} & = \frac{\|S_1(L)\|_F}{\|S_1(P)\|_F} \, \|S_1(N)\|_F \,
\left( \frac{\|S_1 (\Delta K)\|_F}{\|S_1 (\Delta L)\|_F} \right. \\
& \phantom{=} \left.
+ \frac{2 \, \sqrt{k' +1}}{\sigma_{(k' +1 +\ell)m} (S_{k' +1} (M))} \, \frac{\|S_1 (\Delta M)\|_F}{\|S_1 (\Delta L)\|_F}\, (\|S_1(K)\|_F + \|S_1(\Delta K)\|_F)\, \right) \, .
\end{align*}
Finally, the trivial inequalities $\|S_1 (\Delta K)\|_F /\|S_1 (\Delta L)\|_F \leq 1$ and $\|S_1 (\Delta M)\|_F /\|S_1 (\Delta L)\|_F \leq 1$ prove the result.
\end{proof}

Theorem \ref{thm.pertstronglifications} establishes a general framework to analyze when the solution of the complete polynomial eigenvalue problem for $P(\la)$ (including minimal indices) through the use of a backward stable algorithm applied to any of its strong $\ell$-ifications described in Theorem \ref{thm.stronglifications} is backward stable from the polynomial point of view. Such backward stability will be guaranteed whenever $C_{P,L} \approx 1$, or is at least moderate. Clearly, a detailed analysis of $C_{P,L}$ is only possible for specific strong $\ell$-ifications for which $K(\la)$, $M(\la)$, and $N(\la)$ are known. In particular, to determine lower bounds for $\sigma_{(k' +1 +\ell)m} (S_{k' +1} (M))$ will be a necessary challenging task to study the backward stability of specific strong $\ell$-ifications which is beyond the scope of this paper.

\section{Conclusions} \label{sect.conclusions} This paper introduces a new characterization of those polynomial matrices which are minimal bases in terms of the ranks of a finite number of the Sylvester matrices of the considered polynomial matrix.
This characterization is applied to prove in a rigorous sense that polynomial matrices are generically minimal bases with very specific properties encoded in the concept of polynomial matrices with full-Sylvester-rank, which are carefully studied in this paper. In addition, the new characterization permits us to give a necessary and sufficient condition for a minimal basis to be robust under perturbations, which is just that the leading matrix coefficient has full rank, and to determine finite and simple estimations of the size of its robustness neighborhoods. Such results are particularly interesting in the case of full-Sylvester-rank polynomial matrices, since in this case they allow us to study in a very precise way how any minimal basis dual to a given minimal basis with full-Sylvester-rank changes when the given basis is perturbed. Finally some applications of the results of this paper are discussed with detail. We believe that they are just a very few among many other potential applications of this work.


\bibliographystyle{plain}

\end{document}